\documentclass[12pt]{amsart}

\usepackage{latexsym}
\usepackage{amsmath,empheq}
\usepackage{amsfonts}
\usepackage{amssymb}
\usepackage{geometry} 
\usepackage{graphicx,color}
\usepackage{epstopdf}
\usepackage{caption}
\usepackage{diagbox}
\usepackage{subfig}
\usepackage{bm}
\usepackage{algorithm}
\usepackage{algorithmic}
\usepackage{mathrsfs}
\usepackage{tikz}
\usepackage{xfrac}
\usepackage{multirow,float}
\graphicspath{{figures/}}

\DeclareMathOperator{\Div}{div}
\DeclareMathOperator{\sgn}{sgn}
\newtheorem{theorem}{Theorem}[section]
\newtheorem{lemma}[theorem]{Lemma}

\def\l{\langle}
\def\r{\rangle}

\def\bz{\bm{\zeta}}

\def\bn{\mathbf{n}}

\def\O{\Omega}
\def\p{\partial}
\def\LT{{L_2(\O)}}

\def\cT{\mathcal{T}}
\def\cE{\mathcal{E}}

\def\cV{\mathcal{V}}

\def\eps{\varepsilon}
\newcommand{\mce}{\mathcal{E}_h}
\newcommand{\mct}{\mathcal{T}_h}
\newcommand{\bW}{\bm W}
\newcommand{\ds}{\displaystyle}

\newcommand{\bcV}{{\boldsymbol{\mathcal{V}}}}

\newcommand{\bbp}{\mathbb{P}}
\newcommand{\bV}{\bm V}

\newcommand{\avg}[1]{\{{#1}\}}
\newcommand{\cQ}{\mathcal{Q}}
\newcommand{\bl}{\big\langle}
\newcommand{\br}{\big\rangle}

\newcommand{\nab}{\nabla}

\theoremstyle{definition}

\newtheorem{example}[theorem]{Example}

\theoremstyle{remark}
\newtheorem{remark}[theorem]{Remark}

\numberwithin{equation}{section}

\begin{document}

\title[Novel DG for a convection-dominated problem]{Convergence Analysis of Novel Discontinuous Galerkin Methods for a Convection Dominated Problem}

\author{Satyajith Bommana Boyana, Thomas Lewis, Sijing Liu and Yi Zhang}
\address{Satyajith Bommana Boyana, Department of Mathematics and Statistics\\
University Of North Carolina Greensboro\\
Greensboro, NC\\
USA}
\email{s\_bomman@uncg.edu}
\address{Thomas Lewis, Department of Mathematics and Statistics\\
University Of North Carolina Greensboro\\
Greensboro, NC\\
USA}
\email{tllewis3@uncg.edu}
\address{Sijing Liu, The Institute for Computational and Experimental Research in Mathematics\\
Brown University\\
Providence, RI\\
USA}
\email{sijing\_liu@brown.edu}
\address{Yi Zhang, Department of Mathematics and Statistics\\
University Of North Carolina Greensboro\\
Greensboro, NC\\
USA}
\email{y\_zhang7@uncg.edu}

 \date{\today}

\begin{abstract}
In this paper, we propose and analyze a numerically stable and convergent scheme for
a convection-diffusion-reaction equation in the convection-dominated regime. Discontinuous Galerkin (DG) methods are considered since
standard finite element methods for the convection-dominated equation cause spurious oscillations. We choose to follow a novel DG finite element differential calculus framework introduced in Feng et al. (2016) and approximate the infinite-dimensional operators in the equation with the finite-dimensional DG differential operators. Specifically, we construct the numerical method by using the dual-wind discontinuous Galerkin (DWDG) formulation for the diffusive term and the average discrete gradient operator for the convective term along with standard DG stabilization. We prove that the method converges optimally in the convection-dominated regime. Numerical results are provided to support the theoretical findings.
\end{abstract}

\maketitle

\section{Introduction}

Let $\Omega$ be a convex polygonal domain in $\mathbb{R}^2$. We consider the following convection-diffusion-reaction equation

\begin{subequations}\label{eq:cdr}
\begin{alignat}{3}
    -\eps\Delta u+\bz\cdot \nabla u+\gamma u&=f\quad &&\mbox{in}&\quad \O,\\
    u&=g\quad &&\mbox{on}&\quad \partial\O,
\end{alignat}    
\end{subequations}
where the diffusive coefficient $\eps>0$, the source term $f\in\LT$, the convective velocity $\bz\in [W^{1,\infty}(\Omega)]^2$ and the reaction coefficient $\gamma\in W^{1,{\infty}}(\Omega)$ is nonnegative. We assume 
\begin{equation}\label{eq:advassump}
    \gamma-\frac12\nabla\cdot\bz\ge\gamma_0>0
\end{equation}
for some constant $\gamma_0$ so that the problem \eqref{eq:cdr} is well-posed. 
Note that the convective term in \eqref{eq:cdr} is written in non-conservative form. It is equivalent to consider the conservative form 
\begin{subequations}\label{eq:cdr1}
\begin{alignat}{3}
    -\eps\Delta u+\nabla\cdot(\bz u)+(\gamma-\nabla\cdot\bz)u&=f\quad &&\mbox{in}&\quad \O,\\
    u&=g\quad &&\mbox{on}&\quad \partial\O. 
\end{alignat}    
\end{subequations}

The equation \eqref{eq:cdr}/\eqref{eq:cdr1} and the corresponding numerical methods were intensively studied in the literature \cite{rooscdr,guzmandginterior,leykekhman2012local,xu1999monotone,morton1995numerical,johnson2012numerical,stynes2005steady,miller1996fitted,farrell2000robust,lin2018weak} and the references therein. The difficulties of designing numerical methods to solve \eqref{eq:cdr}/\eqref{eq:cdr1} arise when one considers the convection-dominated case, namely, when $0<\varepsilon\ll 1$. In the convection-dominated regime, the solution to \eqref{eq:cdr} exhibits boundary layers near the outflow boundary. We refer to \cite{rooscdr} for more discussion about the analytic behavior of the solution to \eqref{eq:cdr}. The sharp gradients in the boundary layer pose challenges in designing robust numerical methods for \eqref{eq:cdr}. It is well known that a standard finite element method for \eqref{eq:cdr} produces spurious oscillations near the outflow boundary when $\varepsilon\le h\|\bm{\zeta}\|_\infty$, where $h$ is the mesh size of the triangulation. These oscillations then propagate into the interior of the domain where the solution is smooth and destroy the convergence of the finite element methods.

To remedy this issue, many methods were proposed to stabilize the numerical solutions to \eqref{eq:cdr}, for example, SUPG \cite{brooks1982streamline,hughes1986new}, local projection \cite{guermond1999stabilization,knobloch2010generalization,knobloch2009local}, EAFE \cite{xu1999monotone,wang1999crosswind,kim2003multigrid} and DG methods \cite{ayuso2009discontinuous,gopalakrishnan2003multilevel,fu2015analysis,brezzi2004discontinuous}. We refer to \cite{rooscdr,johnson2012numerical,morton1995numerical,farrell2000robust} and the references therein for more details about stabilization techniques. Among these methods, discontinuous Galerkin (DG) methods are favorable in many aspects. First, DG methods do not require the numerical solutions to be continuous, and, hence,  they are more suitable to capture sharp gradients in the solutions. Secondly, DG methods impose the boundary conditions weakly which prevents the boundary layers propagating into the interior of the domain. Lastly, DG methods have a natural upwind stabilization that can stabilize oscillatory behaviors of the numerical solutions \cite{ayuso2009discontinuous,leykekhman2012local}.

In this work, we consider a new type of DG methods inspired by the DG finite element differential calculus framework \cite{feng2016discontinuous}, 
to solve \eqref{eq:cdr}. Specifically, the diffusion part of the equation is discretized by the dual-wind discontinuous Galerkin (DWDG) method and the convection part is discretized by an average discrete divergence operator. 
DWDG methods were introduced for diffusion problems in \cite{lewis2014convergence} based on the DG differential Calculus framework \cite{feng2016discontinuous}. Such methods have optimal convergence properties even in the absence of a penalty term which is different from many existing DG methods. DWDG methods also have been applied to other problems \cite{lewis2023consistency,lewis2020convergence,boyana2023convergence,feng2022dual}. However, the study of the DG finite element differential calculus for convection-diffusion-reaction equations is still missing in the literature. In this paper, we extend the methods to convection-diffusion-reaction equations, with a particular focus on the convection-dominated regime. In order to apply the methods, we first consider the reduced problem \eqref{eq:cr} and approximate the divergence operator $\nabla\cdot$ with the discrete divergence operator $\overline{\Div}_h$. We show, with this choice of discrete operator, the method for the reduced problem \eqref{eq:cr} is consistent with a centered fluxes DG method \cite{di2011mathematical} for the convective term.
This is due to the fact that the discrete operator $\overline{\Div}_h$ is defined as the average of the ``left" discrete divergence operator and the ``right" discrete divergence operator. Using this equivalence, we add a standard penalty term to stabilize the numerical solution which leads to an upwind DG method. Combining the existing DWDG analysis for the diffusive equation with the aforementioned equivalence, we show that the proposed methods are optimal for the convection-diffusion-reaction equations in the sense of the following, 
\begin{equation}
          \|u-u_h\|_{h\sharp}\le\left\{\begin{array}{ll}
              O(h)\quad&\text{if diffusion-dominated},\\
              \\
              O(h^\frac32)\quad&\text{if convection-dominated},
          \end{array}\right.
\end{equation}
where $u$ is the solution to \eqref{eq:cdr}, $u_h$ is the numerical solution, and the mesh-dependent norm $\|\cdot\|_{h\sharp}$ is defined in \eqref{eq:upwnormstrongah}. We analyze the numerical methods using a coercive framework as well as an inf-sup approach. The inf-sup approach allows us to establish a stronger result which also controls the convective derivative (cf. \cite{di2011mathematical}).

The rest of the paper is organized as follows. In Section 2, we recall the results about the DG differential Calculus framework and define various discrete operators that are useful in the following sections. In Section 3, we consider the reduced problem when taking $\eps=0$. We propose the numerical approximations for the reduced problem and establish concrete error estimates. In Section 4, we propose fully discretized methods for \eqref{eq:cdr1} and justify the main convergence theorem. Finally, we provide some numerical results in Section 5 and end with some concluding remarks in Section 6. Some technical proofs are also included in Appendix A.

Throughout this paper, we use $C$
 (with or without subscripts) to denote a generic positive
 constant that is independent of any mesh
 parameter. 
  Also to avoid the proliferation of constants, we use the
   notation $A\lesssim B$ (or $A\gtrsim B$) to
  represent $A\leq \text{(constant)}B$. The notation $A\approx B$ is equivalent to
  $A\lesssim B$ and $B\lesssim A$.

\section{Notations and the DG Differential Calculus} 

In this section, we briefly introduce the DG differential Calculus framework (cf. \cite{feng2016discontinuous}) and the notations that will be used in the rest of the paper. We also provide some useful properties of the DG operators. Throughout the paper we will follow the standard notation for differential operators, function spaces, and norms that can be found, for example, in \cite{BS,Ciarlet}.

\subsection{DG Operators}
Let $W^{m,p}(\O)$ denote the set of all functions that are in $L^p(\O)$ whose weak derivatives up to order $m$ also belong to $L^p(\O)$. We denote $H^m(\O) := W^{m,2}(\O)$ when $p=2$. Let $W^{m,p}_0(\O)$ be the set of functions in $W^{m,p}(\O)$ with vanishing traces up to order $m-1$ on $\partial\O$, and let $H^m_0(\O) = W^{m,2}_0(\O).$ 

Let $\mct$ denote a locally quasi-uniform  simplicial triangulation of $\O \subset \mathbb{R}^2$ with a mesh size $h :=  \max \limits_{T \in \mct} h_T$, where $h_T$ is the diameter of the simplex $T \in \mct$. Let $\mce := \bigcup \limits_{T \in \mct} \partial T$ be the set of all edges in $\mct$ and $\mce^B := \bigcup \limits_{T \in \mct} (\partial T \cap \partial \O)$ be the set of boundary edges in $\mct$.  
Moreover, denote $\mce^I := \mce \setminus \mce^B$ as the set of interior edges in $\mct$. We now define the following piecewise Sobolev spaces 
\begin{equation*}
\begin{aligned}
    W^{m,p}(\mct) &:= \{v : v \vert_T \in W^{m,p}(T) \quad \forall\ T \in \mct \},\\
    \bW^{m,p}(\mct)&:=\{\boldsymbol{v}: \boldsymbol{v} \vert_T \in W^{m,p}(T) \times W^{m,p}(T)  \quad \forall\ T \in \mct  \}.
\end{aligned}
\end{equation*}
We then denote 
\begin{equation}
    \cV_h:=W^{1,1}(\mct)\cap C^0(\mct)\quad \text{and}\quad \bcV_h:=\cV_h \times \cV_h.    
\end{equation}
We also define the following inner products,
\begin{equation}
        (v, w)_{\mct} := \sum \limits_{T \in \mct}\int_T v  w \hspace{0.025in} dx\quad\text{and}\quad \langle v, w \rangle_{\mathcal{S}_h} := \sum \limits_{e \in \mathcal{S}_h}\int_e v  w \hspace{0.025in} ds,
\end{equation}
where $\mathcal{S}_h$ is a subset of $\cE_h$.

Define the DG space
\begin{equation}\label{eq:dgspdef}
    V_h := \{v \in \LT: v \vert_T \in \bbp_1(T) \quad \forall\ T \in \mct \}
\end{equation}
and define $\bV_h:= V_h \times V_h$. Note that $V_h \subset \cV_h$ and $\bV_h \subset \bcV_h$. For each edge $e = \partial T^+ \cap \partial T^-$ with some $T^+$ and $T^-$ in $\cT_h$, we assume the global numbering of $T^+$ is more than that of $T^-$ for simplicity. We define the jump and the average across an edge $e  \in \mce^I$ as follows:
\begin{align*}
    [v]|_e:= v^+ - v^-,\qquad \avg{v}|_e:= \frac12\big( v^+ + v^-\big)\qquad \forall \, v\in \cV_h,
\end{align*}
where $\ds v^{\pm} := v \vert_{T^{\pm}}$. If an edge $e \in \mce^B$, then define  
\begin{align*}
    [v]|_e:= v^+,\qquad \avg{v}|_e:= v^+ \qquad \forall \, v\in \cV_h.
\end{align*}
For an edge $e \in \mce^I$, set $\bn_e = (n_e^{(1)},n_e^{(2)})^t := \bn_{T^+} \vert_{e} = -\bn_{T^-} \vert_{e}$ as the unit normal vector. Given any $v \in \cV_h$, the trace operator $\mathcal{Q}_i^{\pm}$ on $e \in \mce^I$ in the direction $x_i$ $(i=1,2)$ is defined as follows :
\begin{align*}
    \mathcal{Q}_i^+(v) :=
            \begin{cases}
                v \vert_{T^+} , & n_e^{(i)} > 0\\
                v \vert_{T^-} , & n_e^{(i)} < 0\\
                \{v\}, & n_e^{(i)} = 0\\
            \end{cases}
            \hspace{.1in} \text{and} \quad
        \mathcal{Q}_i^-(v) :=
            & \begin{cases}
                v \vert_{T^-} , & n_e^{(i)} > 0\\
                v \vert_{T^+} , & n_e^{(i)} < 0\\
                \{v\}, & n_e^{(i)} = 0\\
            \end{cases}.
\end{align*}

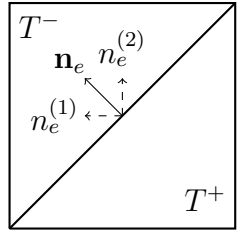
\begin{figure}[ht]
\centering
\begin{tikzpicture}
\draw[thick] plot coordinates {(0,0) (3,0) (3,3) (0,3) (0,0)};
\draw[thick] plot coordinates {(3,3) (0,0) };

\draw[->] plot coordinates {(1.5,1.5) (1,2)};
\draw[dashed,->] plot coordinates {(1.5,1.5) (1,1.5)};
\draw[dashed,->] plot coordinates {(1.5,1.5) (1.5,2)};

\node at (0.8,2.2) {$\mathbf{n}_e$};
\node at (0.6,1.5) {$n^{(1)}_e$};
\node at (1.5,2.4) {$n^{(2)}_e$};

\node at (2.6,0.4) {$T^+$};
\node at (0.4,2.7) {$T^-$};
 
\end{tikzpicture}
\caption{The operators $\mathcal{Q}_i^{\pm}$} \label{figure:qoper}
\end{figure}

See Figure \ref{figure:qoper} for an example where $\mathcal{Q}_1^+(v)=v|_{T^-}$ and $\mathcal{Q}_2^-(v)=v|_{T^+}$. Alternatively, we can define $\mathcal{Q}_i^{\pm}(v) := \avg{v} \pm \frac12\text{sgn}(n_e^{(i)})[v]$, and hence the operators $\mathcal{Q}_i^+(v)$ and $\mathcal{Q}_i^-(v)$ can be interpreted as ``right'' and ``left'' limits in the $x_i$ direction on $e \in \mce^I$. For $e = \partial T^+ \cap \partial \O \in \mce^B$, we simply set $\mathcal{Q}_i^{\pm}(v) := v^+$.

Having defined the trace operators as above, for any $v \in \cV_h$ and a given $g \in L^1(\O)$, we introduce the discrete partial derivatives  $\partial_{h,x_i}^{\pm}, \partial_{h,x_i}^{\pm,g}:\cV_h \rightarrow V_h (i = 1,2)$ as follows:
\begin{subequations}
    \begin{align}
        \bigl(\partial_{h,x_i}^\pm v,\varphi_h\bigr)_{\cT_h}
        &:= \bigl\langle \cQ_i^\pm (v) n^{(i)}, [\varphi_h] \bigr\rangle_{\mce}
        -\bigl(v, \partial_{x_i}\varphi_h \bigr)_{\cT_h}, \label{DGderivativedefinition:1}  \\
        \bigl(\partial_{h,x_i}^{\pm,g} v,\varphi_h\bigr)_{\cT_h}
        &:= \bigl\langle \cQ_i^\pm (v) n^{(i)}, [\varphi_h] \bigr\rangle_{\mce^I} + \bl g n^{(i)},\varphi_h\br_{\mce^B} 
        -\bigl(v, \partial_{x_i}\varphi_h \bigr)_{\cT_h} \label{DGderivativedefinition:2}
    \end{align}
\end{subequations}
for all $\varphi_h\in V_h$.
Accordingly, for any $v \in \cV_h$, the discrete gradient operators are defined as:
\begin{align*}
    \nab_h^\pm v = (\p_{h,x_1}^\pm v, \p_{h,x_2}^\pm v)\quad \text{and} \quad\nab_{h,g}^{\pm} v = (\p_{h,x_1}^{\pm,g} v, \p_{h,x_2}^{\pm,g} v).
\end{align*}
We define the average operators $\overline{\partial}_{h,x_i} v$, $\overline{\partial}_{h,x_i}^{g} v$, $\overline{\nab}_h v$, and$\overline{\nab}_{h,g} v$ as follows,
\begin{align*}
    &\overline{\partial}_{h,x_i} v := \frac12 (\partial_{h,x_i}^+ v + \partial_{h,x_i}^- v),  &&\overline{\partial}^g_{h,x_i} v := \frac12 (\partial_{h,x_i}^{+,g} v + \partial_{h,x_i}^{-,g} v), \\
    &\overline{\nab}_h v := \frac12 (\nab^+_ h v + \nab^-_h v),  &&\overline{\nab}_{h,g} v := \frac12 (\nab_{h,g}^+ v + \nab_{h,g}^- v).
\end{align*}
Similarly, we can also define the discrete divergence operators $\Div_h^{\pm}, \overline{\Div}_h: \bcV_h\rightarrow V_h$ as follows,
\begin{equation}\label{eq:divdef}
    \Div_h^{\pm}\mathbf{v}=\sum_{i=1}^2\partial^{\pm}_{h,x_i}v^{(i)}\quad\text{and}\quad\overline{\Div}_h\mathbf{v}=\frac12(\Div_h^{+}\mathbf{v}+\Div_h^{-}\mathbf{v}).
\end{equation}

\subsection{Preliminary Properties}
We present some preliminary properties of the DG operators defined in the previous subsection and some results that will be used in the subsequent analysis. We first need the following generalized integration by parts formula.
\begin{lemma}\label{lem:ibp}
    For any $v_h\in V_h$ and $\varphi_h\in V_h$, we have
    \begin{equation}
        (\partial_{h,x_i}^{\pm}(\zeta_i v_h), \varphi_h)_{\cT_h}=-(\partial_{h,x_i}^{\mp}(\zeta_i \varphi_h), v_h)_{\cT_h}+(\varphi_h,(\partial_{x_i}\zeta_i)v_h)_{\cT_h}+\l\zeta_i\varphi_h,v_h n^{(i)}\r_{\cE_h^B}.
    \end{equation}
\end{lemma}

\begin{proof}
    By the definitions of $\partial_{h,x_i}^{\pm}$, $\cQ_i^{\pm}$, the fact that  $\zeta_i\in W^{1,\infty}(\Omega)$,  and integration by parts, we have
    \begin{equation}\label{eq:ibpo}
        \begin{aligned}
            (\partial_{h,x_i}^{\pm}(\zeta_i v_h), \varphi_h)_{\cT_h}&=\l \cQ_i^{\pm}(\zeta_i v_h)n^{(i)},[\varphi_h]\r_{\cE_h}-(\zeta_iv_h,\partial_{x_i}\varphi_h)_{\cT_h}\\
            &=\l \zeta_i\cQ_i^{\pm}(v_h)n^{(i)},[\varphi_h]\r_{\cE_h}-(\zeta_iv_h,\partial_{x_i}\varphi_h)_{\cT_h}\\
            &=(\partial_{x_i}(\zeta_iv_h),\varphi_h)_{\cT_h}+\l \zeta_i(\cQ_i^{\pm}(v_h)-\{v_h\}),[\varphi_h]n^{(i)}\r_{\cE_h}\\
            &\quad-\l \zeta_i[v_h],\{\varphi_h\}n^{(i)}\r_{\cE^I_h}\\
            &=(\partial_{x_i}(\zeta_iv_h),\varphi_h)_{\cT_h}\pm\l \frac12\sgn(n^{(i)})[v_h]\zeta_i,[\varphi_h]n^{(i)}\r_{\cE^I_h}\\
            &\quad-\l \zeta_i[v_h],\{\varphi_h\}n^{(i)}\r_{\cE^I_h}\\
            &=(\partial_{x_i}(\zeta_iv_h),\varphi_h)_{\cT_h}+\l \zeta_i[v_h],\pm\frac12|n^{(i)}|[\varphi_h]n^{(i)}-\{\varphi_h\}n^{(i)}\r_{\cE^I_h}.
        \end{aligned}
    \end{equation}
    Use the definition of $\partial_{h,x_i}^{\pm}$ again, we have
    \begin{equation}\label{eq:ibp}
        \begin{aligned}
            (\zeta_i\varphi_h, \partial_{x_i}v_h)_{\cT_h}&=-(\partial_{h,x_i}^{\mp}(\zeta_i \varphi_h), v_h)_{\cT_h}+\l \zeta_i\cQ_i^{\mp}(\varphi_h)n^{(i)},[v_h]\r_{\cE_h}\\
            &=-(\partial_{h,x_i}^{\mp}(\zeta_i \varphi_h), v_h)_{\cT_h}+\l \zeta_i\{\varphi_h\},[v_h]n^{(i)}\r_{\cE^I_h}\\
            &\quad\mp\l\frac12|n^{(i)}|[\varphi_h]\zeta_i, [v_h]\r_{\cE^I_h}+\l \zeta_i\varphi_h, v_hn^{(i)}\r_{\cE^B_h}.
        \end{aligned}
    \end{equation}
    Note that $(\partial_{x_i}(\zeta_iv_h),\varphi_h)_{\cT_h}=(\zeta_i\varphi_h, \partial_{x_i}v_h)_{\cT_h}+(\varphi_h, (\partial_{x_i}\zeta_i) v_h)_{\cT_h}$. Insert \eqref{eq:ibp} into \eqref{eq:ibpo}, we have
    \begin{equation}
        \begin{aligned}
            (\partial_{h,x_i}^{\pm}(\zeta_i v_h), \varphi_h)_{\cT_h}=-(\partial_{h,x_i}^{\mp}(\zeta_i \varphi_h), v_h)_{\cT_h}+(\varphi_h, (\partial_{x_i}\zeta_i) v_h)_{\cT_h}+\l \zeta_i\varphi_h, v_hn^{(i)}\r_{\cE^B_h}.
        \end{aligned}
    \end{equation}
\end{proof}

\begin{remark}
    The immediate consequence of Lemma \ref{lem:ibp} is the following,
    \begin{equation}\label{eq:ibpv}
            (\Div_h^{\pm}(\bm{\zeta} v_h), \varphi_h)_{\cT_h}=-(\Div_h^{\mp}(\bm{\zeta} \varphi_h), v_h)_{\cT_h}+(\varphi_h, (\nabla\cdot\bm{\zeta}) v_h)_{\cT_h}+\l \bm{\zeta}\cdot\bn, v_h\varphi_h\r_{\cE^B_h}.
        \end{equation}    
\end{remark}

\begin{remark}
    Another consequence of the derivation \eqref{eq:ibpo} is the following,
    \begin{equation}\label{eq:consisibp}
        (\overline{\Div}_h(\bm{\zeta} v_h), \varphi_h)_{\cT_h}=(\nabla\cdot(\bm{\zeta}v_h),\varphi_h)_{\cT_h}-\l \bm{\zeta}\cdot\bn[v_h],\{\varphi_h\}\r_{\cE^I_h}.
    \end{equation}
    In fact, the following is also valid, 
    \begin{equation}\label{eq:consisibpc}
        (\overline{\Div}_h(\bm{\zeta} v), \varphi_h)_{\cT_h}=(\nabla\cdot(\bm{\zeta}v),\varphi_h)_{\cT_h}\quad\forall v\in H^1(\O).
    \end{equation}
\end{remark}

\section{The Reduced Problem and Discretization}

Our goal is to design a numerical method based on the DG differential Calculus framework for \eqref{eq:cdr} (or \eqref{eq:cdr1}). Since the DWDG method for the diffusion part is well-established \cite{lewis2014convergence}, we first consider the following reduced problem by taking $\eps=0$,
\begin{subequations}\label{eq:cr}
\begin{alignat}{3}
    \nabla\cdot(\bz u^0)+(\gamma-\nabla\cdot\bz)u^0&=f\quad &&\mbox{in}&\quad \O,\\
    u^0&=g\quad &&\mbox{on}&\quad \partial\O^-,
\end{alignat}    
\end{subequations}
where the inflow part of the boundary $\O^-$ is defined as 
\begin{align*}
    \partial\O^-:=\{x\in\partial\O: \bz(x)\cdot\bn(x)<0\}.
\end{align*}
Here $\bn$ is the outward unit normal vector of $\partial\O$ at $x$.

Let $V^0:=\{v\in\LT\ |\ \bm{\zeta}\cdot \nabla v\in\LT\}$. Then the weak form of the problem \eqref{eq:cr} is to find $u^0\in V^0$ such that
\begin{equation}\label{eq:crweak}
     a^{ar}(u^0,v)=(f,v)_\LT+\int_{\partial\O^-}|\bm{\zeta}\cdot\bn| g\ \!v\ \! dx\quad\forall v\in V^0,
\end{equation}
where the bilinear form $a^{ar}(\cdot,\cdot)$ is defined as
\begin{equation}\label{eq:contbi}
    a^{ar}(v,w)=(\nabla\cdot(\bz v),w)_\LT+((\gamma-\nabla\cdot\bz)v,w)_\LT+\int_{\partial\O^-}|\bm{\zeta}\cdot\bn| v\ \!w\ \! dx.
\end{equation}
The problem \eqref{eq:crweak} is well-posed \cite{di2011mathematical} under the assumption \eqref{eq:advassump}.

The discrete problem for \eqref{eq:crweak} is to find $u^0_h\in V_h$ such that 
\begin{equation}\label{eq:dwdg}
    a_h^{ar}(u^0_h,v_h)=(f,v_h)_\LT+\int_{\partial\O^-}|\bm{\zeta}\cdot\bn| g\ \!v_h\ \! dx\quad\forall v_h\in V_h.
\end{equation}
Here the bilinear form $a^{ar}_h(\cdot,\cdot)$ is defined as,
\begin{equation}\label{eq:discretebi}
    a^{ar}_h(v,w)=(\overline{\Div}_h(\bz v),w)_{\cT_h}+((\gamma-\nabla\cdot\bz)v,w)_\LT+\int_{\partial\O^-}|\bm{\zeta}\cdot\bn| v\ \!w\ \! ds,
\end{equation}
where $\overline{\Div}_h$ is defined in \eqref{eq:divdef}.

\subsection{Consistency}

Let $u^0$ be the solution to \eqref{eq:cr} and $u^0_h$ be the solution to \eqref{eq:dwdg}. We have, by \eqref{eq:consisibpc},
\begin{equation}\label{eq:consis}
\begin{aligned}
     a^{ar}_h(u^0,v_h)&=(\overline{\Div}_h(\bz u^0),v_h)_{\cT_h}+((\gamma-\nabla\cdot\bz)u^0,v_h)_\LT+\int_{\partial\O^-}|\bm{\zeta}\cdot\bn| u^0\ \!v_h\ \! ds\\
     &=(\nabla\cdot(\bz u^0),v_h)_{\cT_h}+((\gamma-\nabla\cdot\bz)u^0,v_h)_\LT+\int_{\partial\O^-}|\bm{\zeta}\cdot\bn| u^0\ \!v_h\ \! ds\\
     &=(\nabla\cdot(\bz u^0)+(\gamma-\nabla\cdot\bz)u^0,v_h)_\LT+\int_{\partial\O^-}|\bm{\zeta}\cdot\bn| u^0\ \!v_h\ \! ds\\
     &=(f,v_h)_\LT+\int_{\partial\O^-}|\bm{\zeta}\cdot\bn| g\ \!v_h\ \! ds\\
     &=a^{ar}_h(u^0_h,v_h)\quad\forall v_h\in V_h.
\end{aligned}
\end{equation}
Therefore we have the usual Galerkin orthogonality
\begin{equation}\label{eq:go}
    a^{ar}_h(u^0-u^0_h,v_h)=0\quad\forall v_h\in V_h.
\end{equation}

\subsection{$L_2$ Coercivity}
Define the norm
\begin{equation}\label{eq:dgl2norm}
    \|v\|_{ar}^2=\|v\|_\LT^2+\int_{\partial\O}\frac12 |\bm{\zeta}\cdot\bn| v^2 ds.
\end{equation}
\begin{lemma}
    We have
    \begin{equation}\label{eq:advcoer}
        a^{ar}_h(v_h,v_h)\ge C\|v_h\|_{ar}^2\quad\forall v_h\in V_h.
    \end{equation}
\end{lemma}
\begin{proof}
    It follows from \eqref{eq:ibpv} that
    \begin{equation}
        \begin{aligned}
            a^{ar}_h(v_h,v_h)&=(\overline{\Div}_h(\bz v_h),v_h)_{\cT_h}+((\gamma-\nabla\cdot\bz)v_h,v_h)_\LT+\int_{\partial\O^-}|\bm{\zeta}\cdot\bn| v_h^2\ \! ds\\
            &=\frac12(\Div^+_h(\bz v_h),v_h)_{\cT_h}+\frac12(\Div^-_h(\bz v_h),v_h)_{\cT_h}\\
            &\quad\quad+((\gamma-\nabla\cdot\bz)v_h,v_h)_\LT+\int_{\partial\O^-}|\bm{\zeta}\cdot\bn| v_h^2\ \! ds\\
            &=\frac12(v_h, (\nabla\cdot\bm{\zeta}) v_h)_{\cT_h}+\frac12\l \bm{\zeta}\cdot\bn, v_h^2\r_{\cE^B_h}+\\
            &\quad\quad+((\gamma-\nabla\cdot\bz)v_h,v_h)_\LT+\int_{\partial\O^-}|\bm{\zeta}\cdot\bn| v_h^2\ \! ds\\
            &=((\gamma-\frac12\nabla\cdot\bz)v_h,v_h)_\LT+\int_{\partial\O}\frac12 |\bm{\zeta}\cdot\bn| v_h^2 ds.
        \end{aligned}
    \end{equation}
    The estimate \eqref{eq:advcoer} follows from the assumption \eqref{eq:advassump} immediately.
\end{proof}

\begin{remark}\label{rem:cf}
    Using \eqref{eq:consisibp}, we obtain
    \begin{equation}\label{eq:cf}
    \begin{aligned}
        a^{ar}_h(u^0_h,v_h)&=(\nabla\cdot(\bm{\zeta}u^0_h),v_h)_{\cT_h}-\l \bm{\zeta}\cdot\bn[u^0_h],\{v_h\}\r_{\cE^I_h}\\
        &\quad+((\gamma-\nabla\cdot\bz)u^0_h,v_h)_\LT+\int_{\partial\O^-}|\bm{\zeta}\cdot\bn| u^0_h\ \!v_h\ \! ds.
    \end{aligned}
    \end{equation}
    Therefore, the proposed method \eqref{eq:dwdg} is consistent with the standard centered fluxes DG method (cf. \cite{di2011mathematical}).
\end{remark}

\subsection{Stabilization}

It is well-known that the solution to \eqref{eq:cf} (or equivalently, \eqref{eq:dwdg}) exhibits spurious oscillations near the outflow boundary if no additional stabilization is added. Hence, we define the following method with a stabilization term: Find $u^0_h\in V_h$ such that

\begin{equation}\label{eq:dwdgstable}
    a^{upw}_h(u^0_h,v_h)=(f,v_h)_\LT+\int_{\partial\O^-}|\bm{\zeta}\cdot\bn| g\ \!v_h\ \! dx\quad\forall v_h\in V_h,
\end{equation}
where the bilinear form $a^{upw}_h(\cdot,\cdot)$ is defined as,
\begin{equation}\label{eq:discretebistab}
    a^{upw}_h(v,w)=a^{ar}_h(v,w)+\l \frac12|\bm{\zeta}\cdot\bn|[v],[w]\r_{\cE^I_h}.
\end{equation}

\begin{remark}
    The method \eqref{eq:dwdgstable} can be interpreted as an upwind method \cite{di2011mathematical}.
\end{remark}

It is trivial to see that \eqref{eq:dwdgstable} is a consistent method in the sense that
\begin{equation}\label{eq:stabconsis}
    a_h^{upw}(u^0-u^0_h,v_h)=0 \quad\forall v_h\in V_h.
\end{equation}
Define the norm $\|\cdot\|_{upw}$ on $V_h$ as
\begin{equation}\label{eq:upwnorm}
    \|v\|^2_{upw}=\|v\|_{ar}^2+\sum_{e\in\cE_h^I}\int_e \frac12|\bm{\zeta}\cdot\bn|[v]^2\ \!ds.
\end{equation}
\begin{lemma}
    We have, for all $v\in V_h$,
    \begin{equation}\label{eq:stabcoer}
        a^{upw}_h(v_h,v_h)\ge C\|v_h\|^2_{upw}.
    \end{equation}
\end{lemma}
\begin{proof}
    The coercivity follows from \eqref{eq:advcoer}, \eqref{eq:discretebistab}, and \eqref{eq:upwnorm} immediately.
\end{proof}

\subsection{Convergence Analysis}

We would like to establish the error estimates of the stabilized method \eqref{eq:dwdgstable}. Note that \eqref{eq:dwdgstable} is well-posed due to the discrete coercivity \eqref{eq:stabcoer}. We define a stronger norm on $V^0+V_h$,
\begin{equation}\label{eq:upwnorm1}
    \|v\|^2_{upw,*}=\|v\|^2_{upw}+\sum_{T\in\cT_h}\|v\|^2_{L_2(\partial T)}.
\end{equation}
It can be shown \cite{di2011mathematical} that for all $v\in V^0$ and $w_h\in V_h$, we have
\begin{equation}\label{eq:stabbound}
    a_h^{upw}(v-\pi_h v,w_h)\le C\|v-\pi_h v\|_{upw,*}\|w_h\|_{upw},   
\end{equation}
where $\pi_h: V^0\rightarrow V_h$ is the $L_2$-orthogonal projection. 
Combining \eqref{eq:stabconsis}, \eqref{eq:stabcoer}, and \eqref{eq:stabbound}, we conclude
\begin{equation}\label{eq:abesti}
    \|u^0-u^0_h\|_{upw}\le C\|u^0-\pi_hu^0\|_{upw,*}.
\end{equation}
By standard projection error estimates, we have (cf. \cite{di2011mathematical}),
\begin{theorem}\label{thm:upwesti}
    Let $u^0$ be the solution to \eqref{eq:crweak} and $u^0_h$ be the solution to \eqref{eq:dwdgstable}. Assume $u^0\in H^2(\O)$ and then we have
    \begin{equation}
        \|u^0-u^0_h\|_{upw}\le Ch^{\frac32}\|u^0\|_{H^2(\O)}.
    \end{equation}
\end{theorem}

\subsection{Convergence Analysis Based On an inf-sup Condition}\label{sec:infsupconv}

We could obtain a similar error estimate with a stronger norm which involves the gradient in the direction of $\bm{\zeta}$. Define 
\begin{equation}\label{eq:upwnormstrong}
    \|v\|^2_{upw\sharp}=\|v\|^2_{upw}+\sum_{T\in\cT_h}h_T\|\bm{\zeta}\cdot\nabla v\|^2_{L_2(T)}.
\end{equation}
We first need the following inf-sup condition (cf. \cite{di2011mathematical}).
\begin{lemma}\label{lem:infsupconv}
    We have
    \begin{equation}\label{eq:disinfsup}
        \sup_{w_h\in V_h\setminus\{0\}}\frac{a_h^{upw}(v_h,w_h)}{\|w_h\|_{upw\sharp}}\ge C\|v_h\|_{upw\sharp}\quad \forall v_h\in V_h,
    \end{equation}
    where the constant $C$ is independent of $\bm{\zeta}$ and $h$.
\end{lemma}

\begin{proof}
    The proof is identical to that of \cite[Lemma 2.35]{di2011mathematical} due to Remark \ref{rem:cf}.
\end{proof}

To formulate an abstract error estimate, we define the following norm on $V^0+V_h$,

\begin{equation}\label{eq:upwnormss}
    \|v\|^2_{upw\sharp,*}=\|v\|^2_{upw\sharp}+\sum_{T\in\cT_h}(h_T^{-1}\|v\|^2_{L_2(T)}+\|v\|^2_{L_2(\partial T)}).
\end{equation}
Similar to \eqref{eq:stabbound}, we have (cf. \cite{di2011mathematical}),
\begin{equation}\label{eq:stabboundss}
    a_h^{upw}(v-\pi_h v,w_h)\le C\|v-\pi_h v\|_{upw\sharp,*}\|w_h\|_{upw\sharp}
\end{equation}
for all $v\in V^0$ and $w_h\in V_h$.

The immediate consequence of \eqref{eq:disinfsup} and \eqref{eq:stabboundss} is the following lemma.
\begin{lemma}\label{lem:conestiuwss}
Let $u^0$ be the solution to \eqref{eq:crweak} and $u^0_h$ be the solution to \eqref{eq:dwdgstable}. Assume $u^0\in H^2(\O)$ and then we have
     \begin{equation}\label{eq:conestiuwss}
    \|u^0-u^0_h\|_{upw\sharp}\le C\|u^0-\pi_hu^0\|_{upw\sharp,*}\le Ch^{\frac32}\|u^0\|_{H^2(\O)}.
\end{equation}
 \end{lemma} 

 \begin{proof}
     It follows from \eqref{eq:disinfsup}, \eqref{eq:stabboundss}, and \eqref{eq:stabconsis} that,
     \begin{equation}\label{eq:uwinfsupesti}
     \begin{aligned}
             \|\pi_hu^0-u^0_h\|_{upw\sharp}&\le C\sup_{w_h\in V_h\setminus\{0\}}\frac{a_h^{upw}(\pi_hu^0-u^0_h,w_h)}{\|w_h\|_{upw\sharp}}\\
             &=C\sup_{w_h\in V_h\setminus\{0\}}\frac{a_h^{upw}(\pi_hu^0-u^0,w_h)}{\|w_h\|_{upw\sharp}}\\
             &\le C\|u^0-\pi_hu^0\|_{upw\sharp,*}.
     \end{aligned}
     \end{equation}
     The first inequality in \eqref{eq:conestiuwss} is immediate due to triangle inequality and \eqref{eq:uwinfsupesti}. For the second inequality, we have
     \begin{equation}\label{eq:convdxproj}
         \sum_{T\in\cT_h}h_T\|\bm{\zeta}\cdot\nabla (u^0-\pi_hu^0)\|^2_{L_2(T)}\le C h^3\|u^0\|^2_{H^2(\O)}
     \end{equation}
     and
     \begin{equation}\label{eq:uwssl2esti}
         \sum_{T\in\cT_h}h_T^{-1}\|u^0-\pi_hu^0\|^2_{L_2(T)}\le C h^3\|u^0\|^2_{H^2(\O)}
     \end{equation}
     by standard projection error estimates. We finish the proof by combining Theorem \ref{thm:upwesti}, \eqref{eq:upwnormstrong}, \eqref{eq:upwnormss}, \eqref{eq:convdxproj}, and \eqref{eq:uwssl2esti}.
 \end{proof}

\section{The Problem \eqref{eq:cdr1} and Discretization}
The weak form of the problem \eqref{eq:cdr1} is to find $u\in V:=H^1(\O)$ such that
\begin{equation}\label{eq:cdr1weak}
     a(u,v)=(f,v)_\LT+\int_{\partial\O^-}|\bm{\zeta}\cdot\bn| g\ \!v\ \! dx\quad\forall v\in V,
\end{equation}
where the bilinear form $a(\cdot,\cdot)$ is defined as
\begin{equation}\label{eq:contbicdr1}
    a(v,w)=\eps a^d(v,w)+a^{ar}(v,w)
\end{equation}
for the bilinear form $a^d(v,w):=(\nabla v, \nabla w)_\LT$ and $a^{ar}(\cdot,\cdot)$ defined in \eqref{eq:contbi}. The problem \eqref{eq:cdr1weak} is well-posed \cite{di2011mathematical} under the assumption \eqref{eq:advassump}.

The discrete problem for \eqref{eq:cdr1weak} is to find $u_h\in V_h$ such that 
\begin{equation}\label{eq:dwdgf}
\begin{aligned}
    a_h(u_h,v_h)&=(f,v_h)_\LT+\int_{\partial\O^-}|\bm{\zeta}\cdot\bn| g\ \!v_h\ \! dx\\
    &\quad\quad-\eps\l g,\overline{\nabla}_{h,0} v_h\cdot\bn-\frac{\sigma_e}{h_e}v_h\r_{\cE_h^B}\quad\forall v_h\in V_h,
\end{aligned}
\end{equation}
where the bilinear form $a_h(\cdot,\cdot):=\eps a^{d}_h(\cdot,\cdot)+a^{upw}_h(\cdot,\cdot)$. Here the bilinear form $a_h^{upw}(\cdot,\cdot)$ is defined in \eqref{eq:discretebistab} and $a_h^d(\cdot,\cdot)$ (cf. \cite{lewis2020convergence}) is defined as

\begin{equation}\label{eq:dwdgdfbilinear}
    \begin{aligned}
        a_h^d(v,w):=\frac12\left[(\nabla^+_{h,0}v,\nabla^+_{h,0}w)_{\cT_h}+(\nabla^-_{h,0}v,\nabla^-_{h,0}w)_{\cT_h}\right]+\l\frac{\sigma_e}{h_e}[v],[w]\r_{\cE_h}
    \end{aligned}
\end{equation}
with the penalty parameter $\sigma_e\ge0$ for all $e\in \cE_h$.

\begin{remark}
    Unlike most standard DG methods where the penalty parameter is positive, DWDG methods allow $\sigma_e=0$ for all $e\in\cE_h$ under the assumptions $\mathcal{T}_h$ is locally quasi-uniform and each simplex in the triangulation has at most one boundary edge. 
    This result was established in \cite{lewis2014convergence,lewis2020convergence,feng2016discontinuous} for a diffusive equation. Here we maintain the same assumptions and allow the case where $\sigma_e=0$ for the general convection-diffusion-reaction equation. 
\end{remark}

\subsection{Consistency}

Let $u$ be the solution to \eqref{eq:cdr1} and $u_h$ be the solution to \eqref{eq:dwdgf}. It follows from \cite{lewis2020convergence} and \eqref{eq:stabconsis} that

\begin{equation}\label{eq:dwdgconsis}
    a_h(u-u_h,v_h)=-\eps\l \{\overline{\nabla}_{h,g}u-\nabla u\}\cdot\bn,[v_h]\r_{\cE_h}\quad\forall v_h\in V_h.
\end{equation}

\begin{remark}
    The method \eqref{eq:dwdgf} is not consistent in the sense of \eqref{eq:dwdgconsis}.
\end{remark}

\subsection{Coercivity}

Define the norm $\|\cdot\|_h$ on $V_h$ by
\begin{equation}\label{eq:dghnorm}
    \|v\|_h^2:=\eps\|v\|_d^2+\|v\|_{upw}^2,
\end{equation}
where $\|\cdot\|_{upw}$ is defined in \eqref{eq:upwnorm} and $\|\cdot\|_d$ is defined as
\begin{equation}\label{eq:dwdgnorm}
    \|v\|_d^2:=\frac12(\|\nabla^+_{h,0}v\|_{\LT}^2+\|\nabla^-_{h,0}v\|_{\LT}^2)+\sum_{e\in\cE_h}\frac{\sigma_e}{h_e}\|[v]\|_{L_2(e)}^2.
\end{equation}

It follows from \eqref{eq:stabcoer}, \eqref{eq:dwdgdfbilinear}, and \eqref{eq:dwdgnorm} that
\begin{equation}\label{eq:ahcoer}
    a_h(v_h,v_h)\ge C\|v_h\|_h^2 \quad\forall v_h\in V_h.
\end{equation}

\subsection{Convergence Analysis}
Note that \eqref{eq:dwdgf} is well-posed due to the discrete coercivity \eqref{eq:ahcoer}.
It is shown (cf. \cite[(3.15)]{lewis2020convergence}) that for $\sigma_e\ge0$, 
\begin{equation}\label{eq:dwdgbounded}
  a_h^d(v,w)\le\|v\|_d\|w\|_d\quad\forall v, w\in V+V_h.
\end{equation}
Consequently, we have, by \eqref{eq:stabbound} and \eqref{eq:dwdgbounded},
\begin{equation}\label{eq:ahbound}
    a_h(v-\pi_hv,w_h)\le C\|v-\pi_hv\|_{h,*}\|w_h\|_h\quad\forall v\in V,\ w_h\in V_h,
\end{equation}
where the norm $\|\cdot\|_{h,*}$ is defined by
\begin{equation}\label{eq:hstarnorm}
    \|v\|_{h,*}^2=\eps\|v\|_d^2+\|v\|^2_{upw,*}.
\end{equation}
Here the operator $\pi_h: V\rightarrow V_h$ is the $L_2$-orthogonal projection. 

\begin{theorem}
    Let $u$ be the solution to \eqref{eq:cdr1} and $u_h$ be the solution to \eqref{eq:dwdgf}. Assume $u\in H^2(\O)$ and then we have
    \begin{equation}\label{eq:hnormesti}
        \|u-u_h\|_h\le C(\eps^\frac12h+h^\frac32)\|u\|_{H^2(\O)}.
    \end{equation}
\end{theorem}

\begin{proof}
    It follows from \cite[Theorem 4.2]{lewis2020convergence} and \cite{di2011mathematical} that
    \begin{equation}\label{eq:projerror}
    \begin{aligned}
         \|u-\pi_hu\|_h^2&=\eps\|u-\pi_hu\|^2_d+\|u-\pi_hu\|^2_{upw}\\
        &\le C  (\eps h^2+h^3)\|u\|^2_{H^2(\O)}.
    \end{aligned}
    \end{equation}
    It follows from \eqref{eq:ahcoer}, \eqref{eq:dwdgconsis}, and \eqref{eq:ahbound} that
    \begin{equation}\label{eq:consiserror}
    \begin{aligned}
        \|\pi_hu-u_h\|_h^2&\lesssim a_h(\pi_hu-u_h,\pi_hu-u_h)\\
        &= a_h(\pi_hu-u,\pi_hu-u_h)+a_h(u-u_h,\pi_hu-u_h)\\
        &\lesssim\|u-\pi_hu\|_{h,*}\|\pi_hu-u_h\|_h\\
        &\quad-\eps\l \{\overline{\nabla}_{h,g}u-\nabla u\}\cdot\bn,[\pi_hu-u_h]\r_{\cE_h}.
    \end{aligned}
    \end{equation}
    It is shown in \cite{lewis2020convergence} that
    \begin{equation}\label{eq:consiserror1}
        \Big\vert\l \{\overline{\nabla}_{h,g}u-\nabla u\}\cdot\bn,[\pi_hu-u_h]\r_{\cE_h}\Big\vert\le Ch\|u\|_{H^2(\O)}\|\pi_hu-u_h\|_d,
    \end{equation}
    and hence
    \begin{equation}\label{eq:consiserror3}
        \eps\Big\vert\l \{\overline{\nabla}_{h,g}u-\nabla u\}\cdot\bn,[\pi_hu-u_h]\r_{\cE_h}\Big\vert\le C\eps^\frac12h\|u\|_{H^2(\O)}\|\pi_hu-u_h\|_h.
    \end{equation}
    Similar to Theorem \ref{thm:upwesti}, we have, by the trace inequality with scaling,
    \begin{equation}\label{eq:bderror}
    \begin{aligned}
        \sum_{T\in\cT_h}\|u-\pi_hu\|^2_{L_2(\partial T)}&\le C\sum_{T\in\cT_h}h_T^{-1}\|u-\pi_hu\|^2_{L_2(T)}+h_T\|\nabla (u-\pi_hu)\|^2_{L_2(T)}\\
        &\le C h^3\|u\|_{H^2(\O)}.
    \end{aligned}
    \end{equation}
    It follows from \eqref{eq:projerror}, \eqref{eq:hstarnorm}, and \eqref{eq:upwnorm1} that
    \begin{equation}\label{eq:hstarnormperror}
        \|u-\pi_hu\|_{h,*}\le C(\eps^\frac12 h+h^\frac32)\|u\|_{H^2(\O)}.
    \end{equation}
    We then conclude, by \eqref{eq:hstarnormperror}, \eqref{eq:consiserror}, and \eqref{eq:consiserror3},
    \begin{equation}\label{eq:consiserror2}
        \|\pi_hu-u_h\|_h\le C (\eps^\frac12 h+h^\frac32)\|u\|_{H^2(\O)}.
    \end{equation}
    Combining \eqref{eq:projerror}, \eqref{eq:consiserror2}, and triangle inequality, we obtain
    \begin{equation}
        \|u-u_h\|_h\le C (\eps^\frac12 h+h^\frac32)\|u\|_{H^2(\O)}.
    \end{equation}
\end{proof}

\subsection{Convergence Analysis Based On an inf-sup Condition}

We present an error estimate with a stronger norm that is similar to Section \ref{sec:infsupconv}. Define 
\begin{equation}\label{eq:upwnormstrongah}
    \|v\|^2_{h\sharp}=\|v\|^2_{h}+\sum_{T\in\cT_h}h_T\|\bm{\zeta}\cdot\nabla v\|^2_{L_2(T)}.
\end{equation}
We first need the following inf-sup condition (cf. \cite[Lemma 2.35]{di2011mathematical} and \cite[Lemma A.1]{gopalakrishnan2003multilevel}).
\begin{lemma}\label{lem:infsupconvah}
    We have
    \begin{equation}\label{eq:disinfsuph}
        \sup_{w_h\in V_h\setminus\{0\}}\frac{a_h(v_h,w_h)}{\|w_h\|_{h\sharp}}\ge C\|v_h\|_{h\sharp}.
    \end{equation}
\end{lemma}

\begin{proof}
    A proof is provided in Appendix \ref{apx:infsup}.
\end{proof}

Similar to Section \ref{sec:infsupconv}, we define the following norm on $V+V_h$,

\begin{equation}\label{eq:upwnormssah}
    \|v\|^2_{h\sharp,*}:=\|v\|^2_{h\sharp}+\sum_{T\in\cT_h}(h_T^{-1}\|v\|^2_{L_2(T)}+\|v\|^2_{L_2(\partial T)}).
\end{equation}

We then have \cite{di2011mathematical}
\begin{equation}\label{eq:stabboundssh}
    a_h(v-\pi_h v,w_h)\le C\|v-\pi_h v\|_{h\sharp,*}\|w_h\|_{h\sharp},
\end{equation}
for all $v\in V$ and $w_h\in V_h$.

Consequently, we obtain the following convergence theorem from \eqref{eq:disinfsuph} and \eqref{eq:stabboundssh}.
\begin{theorem}\label{thm:conestiuwssh}
Let $u$ be the solution to \eqref{eq:cdr1} and $u_h$ be the solution to \eqref{eq:dwdgf}. Assume $u\in H^2(\O)$ and then we have
     \begin{equation}\label{eq:conestiuwssh}
    \|u-u_h\|_{h\sharp}\le C (\eps^\frac12 h+h^\frac32)\|u\|_{H^2(\O)}.
\end{equation}
 \end{theorem} 

 \begin{proof}
 It follows from \eqref{eq:dwdgconsis}, \eqref{eq:disinfsuph}, \eqref{eq:stabboundssh}, and \eqref{eq:consiserror3} that
     \begin{equation}\label{eq:uwinfsupestih}
     \begin{aligned}
             \|\pi_hu-u_h\|_{h\sharp}&\le C\sup_{w_h\in V_h\setminus\{0\}}\frac{a_h(\pi_hu-u_h,w_h)}{\|w_h\|_{h\sharp}}\\
             &=C\sup_{w_h\in V_h\setminus\{0\}}\frac{a_h(\pi_hu-u,w_h)+a_h(u-u_h, w_h)}{\|w_h\|_{h\sharp}}\\
             &\le C\|u-\pi_hu\|_{h\sharp,*}+C\eps^\frac12h\|u\|_{H^2(\O)}.
     \end{aligned}
     \end{equation}
     We also obtain the following estimate, by \eqref{eq:hstarnormperror} and arguments similar to \eqref{eq:convdxproj} and \eqref{eq:uwssl2esti},
     \begin{equation}\label{eq:projhsharps}
        \|u-\pi_hu\|_{h\sharp,*}\le C(\eps^\frac12 h+h^\frac32)\|u\|_{H^2(\O)}.         
     \end{equation}
     We finish the proof by combining \eqref{eq:uwinfsupestih}, \eqref{eq:projhsharps}, and triangle inequality.
 \end{proof}

 \begin{remark}\label{rem:convergence}
     Theorem \ref{thm:conestiuwssh} implies the following convergence results,
     \begin{equation}
          \|u-u_h\|_{h\sharp}\le\left\{\begin{array}{ll}
              O(h)\quad&\text{when \eqref{eq:cdr1} is diffusion-dominated},\\
              \\
              O(h^\frac32)\quad&\text{when \eqref{eq:cdr1} is convection-dominated}.
          \end{array}\right.
      \end{equation}
      Note that $\|u\|_{H^2(\O)}=O(\eps^{-\frac32})$ (\cite[Part III, Lemma 1.18]{rooscdr}), hence, the estimate \eqref{eq:conestiuwssh} is not informative when $\eps\le h$. More delicate interior error estimates that stay away from the boundary layers and interior layers for standard DG methods can be found in \cite{leykekhman2012local,guzmandginterior}.
 \end{remark}

 \section{Numerical Results}\label{sec:numer}

In this section, we present some numerical examples that support the theoretical results. All experiments are performed using MATLAB. We measure the absolute errors both globally and locally to examine the local behaviors of our numerical methods. For comparison, we test the penalty parameter choices $\sigma_e=5$ and $\sigma_e=0$ for all $e\in\cE_h$.

\begin{example}[Smooth Solution]\label{Ex:smsol}
In this example, we take $\O=[1,3] \times [0,2]$, $\gamma = 0$, $\bm{\zeta} = [ x_1, x_2]^t$, and $\eps = 10^{-9}.$ We define the exact solution as 
\begin{align}\label{eq:smsol}
    u(x_1,x_2) = \dfrac{x_2}{x_1}.
\end{align}    
\end{example}

\begin{figure}[H]
    \centering
    \caption{Results of Example \ref{Ex:smsol}: $u_h$ (left) and $u$ (right) for $h = \sfrac{1}{128}$.}
    \label{fig:smsol}
    \begin{minipage}{.365\textwidth}
    \hspace{-0.7in}
      \includegraphics[width=1.30\linewidth]{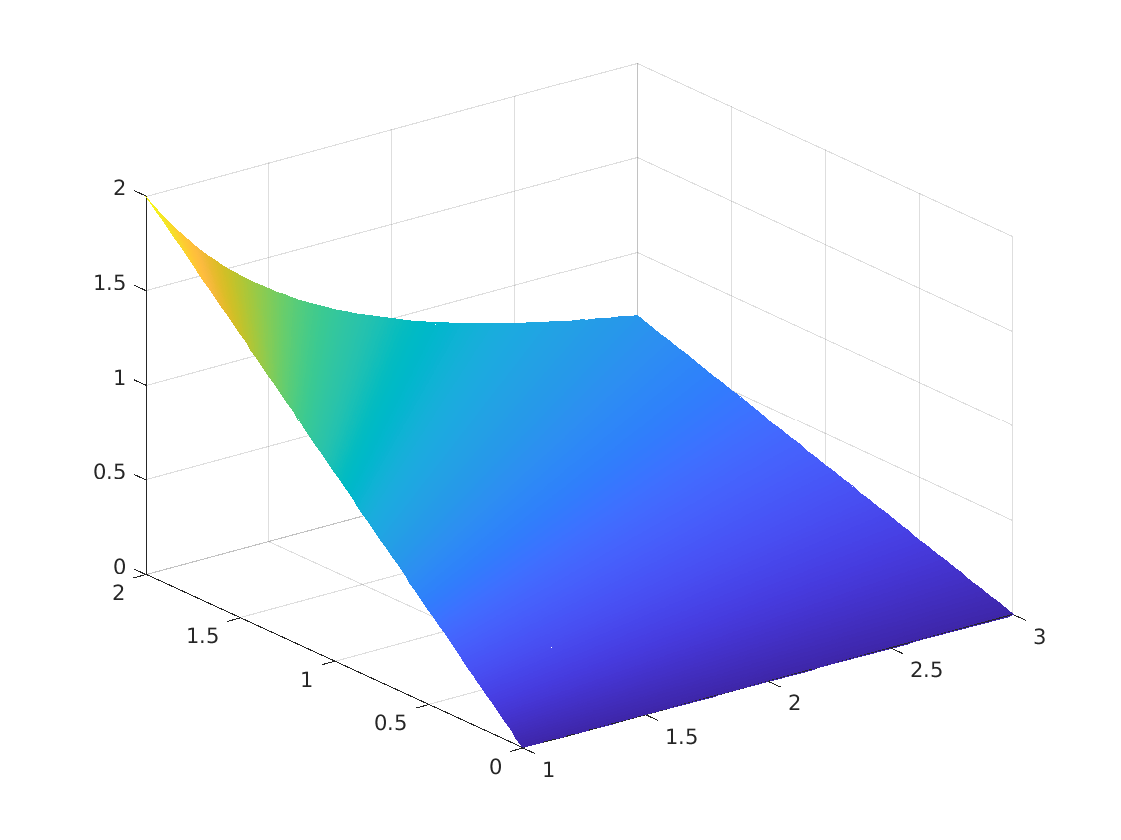}
    \end{minipage}%
    \begin{minipage}{.365\textwidth}
    \centering
      \includegraphics[width=1.30\linewidth]{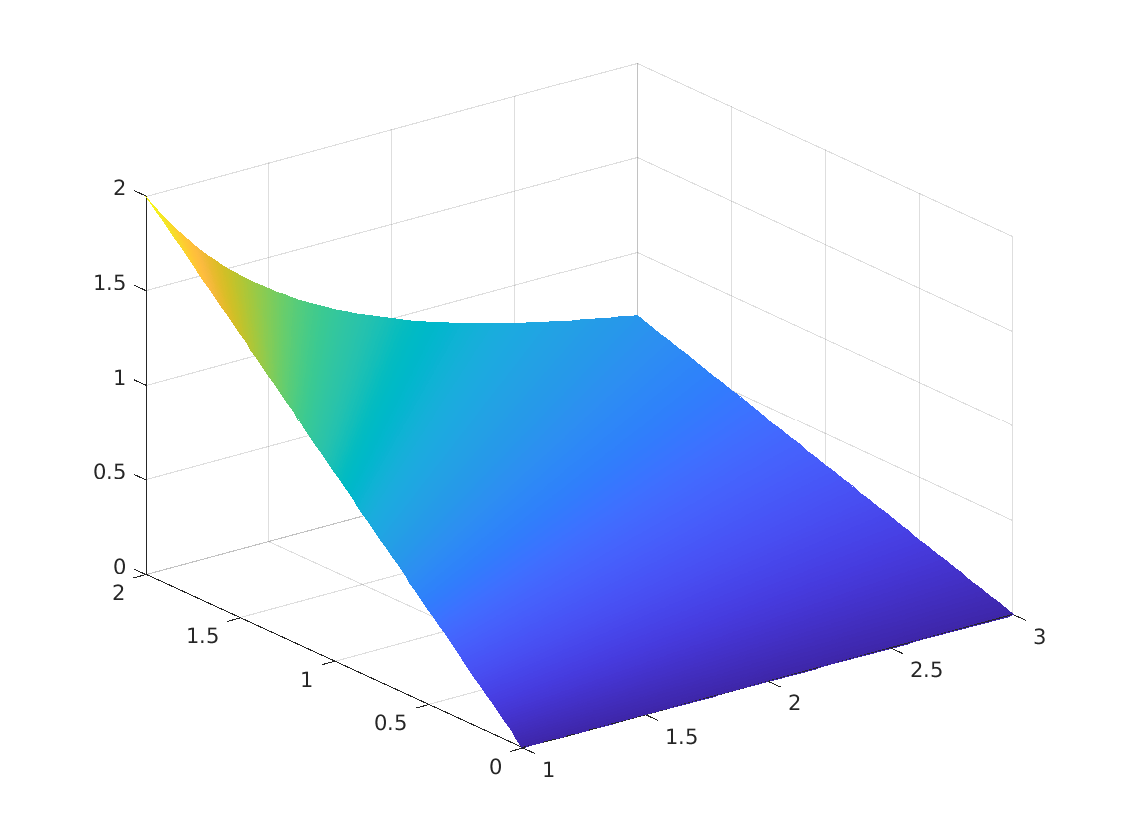}
    \end{minipage}
\end{figure}

We show the global convergence rates in Table \ref{Table:smsol}. We observe $O(h^2)$ convergence in the $L_2$ norm and $O(h^\frac32)$ convergence in the $\|\cdot\|_h$ 
 and $\|\cdot\|_{h\sharp}$ norms. See Figure \ref{fig:smsol} for an illustration of the numerical solution and the exact solution. Note that the convergence rates are all optimal. Indeed, the optimal convergence rates in $L^2$ norm is due to the smoothness of the solution, similar convergence behavior was observed in \cite{ayuso2009discontinuous}. The optimal convergence rates in the $\|\cdot\|_h$ 
 and $\|\cdot\|_{h\sharp}$ norms match with our theoretical results in Remark \ref{rem:convergence}.

\begin{example}[Boundary Layer \cite{ayuso2009discontinuous}]\label{ex:bl1}
In this example, we take $\O= [0,1]^2$, $\gamma = 0$, $\bm{\zeta} = [1,1]^t$ , and $\eps = 10^{-9}$. We define the exact solution as

\begin{align}
    u(x_1,x_2) = x_1 + x_2(1-x_1) + \frac{e^{-1/\eps}-e^{\frac{(x_1-1)(1-x_2)}{\eps}}}{1-e^{-1/\eps}}.
    \label{eq:ex1sol}
\end{align}
Note that the exact solution exhibits boundary layers near $x_1=1$ and $x_2=1$.
\end{example}

\begin{table}[H]
    \centering
    \begin{tabular}{|c|c|cc|cc|cc|}
        \hline
        \multicolumn{1}{|c|}{}  
        & \multicolumn{1}{c|}{}
        & \multicolumn{2}{c|}{$L_2$}
        & \multicolumn{2}{c|}{$||\cdot||_h$}
        & \multicolumn{2}{c|}{$||\cdot||_{h\sharp}$}\\
        & $h$ & Error & Rate & Error & Rate & Error & Rate \\ 
        \hline
        \multirow{7}{*}{$\sigma_e=0$} 
        & 1/4  & 9.95e-03 & - & 3.35e-02 & - &  5.85e-02 & -     \\
        & 1/8  & 2.41e-03 & 2.04 & 1.15e-02 & 1.54 &  2.13e-02 & 1.45      \\
        & 1/16 & 6.01e-04 & 2.01 & 3.97e-03 & 1.53 &  7.65e-03 & 1.48      \\
        & 1/32 & 1.50e-04 & 2.00 & 1.38e-03 & 1.52 &  2.72e-03 & 1.49      \\
        & 1/64 & 3.73e-05 & 2.01 & 4.86e-04 & 1.51 &  9.64e-04  & 1.50      \\
        \hline
        \hline
            \multirow{7}{*}{$\sigma_e=5$} 
        & 1/4  & 9.95e-03 & - & 3.35e-02 & - &  5.85e-02 & -      \\
        & 1/8  & 2.41e-03 & 2.04 & 1.15e-02 & 1.54 &  2.13e-02 & 1.45      \\
        & 1/16 & 6.01e-04 & 2.01 & 3.97e-03 & 1.53 &  7.65e-03 & 1.48      \\
        & 1/32 & 1.50e-04 & 2.00 & 1.38e-03 & 1.52 &  2.72e-03 & 1.49      \\
        & 1/64 & 3.73e-05 & 2.01 & 4.86e-04 & 1.51 &  9.64e-04  & 1.50      \\
        \hline
    \end{tabular}
    \caption{Errors and rates of convergence for $u_h$ on $\O = [1,3] \times [0,2]$ for Example \ref{Ex:smsol} when $\eps = 10^{-9}$.}
    \label{Table:smsol}
\end{table}

\begin{figure}[h]
    \centering
    \caption{Results of Example \ref{ex:bl1}: $u_h$ (left) and $u$ (right) for $\eps = 10^{-9}$, $h = \sfrac{1}{128}$}
    \label{fig:ex1}
    \begin{minipage}{.365\textwidth}
    \hspace{-0.7in}
      \includegraphics[width=1.30\linewidth]{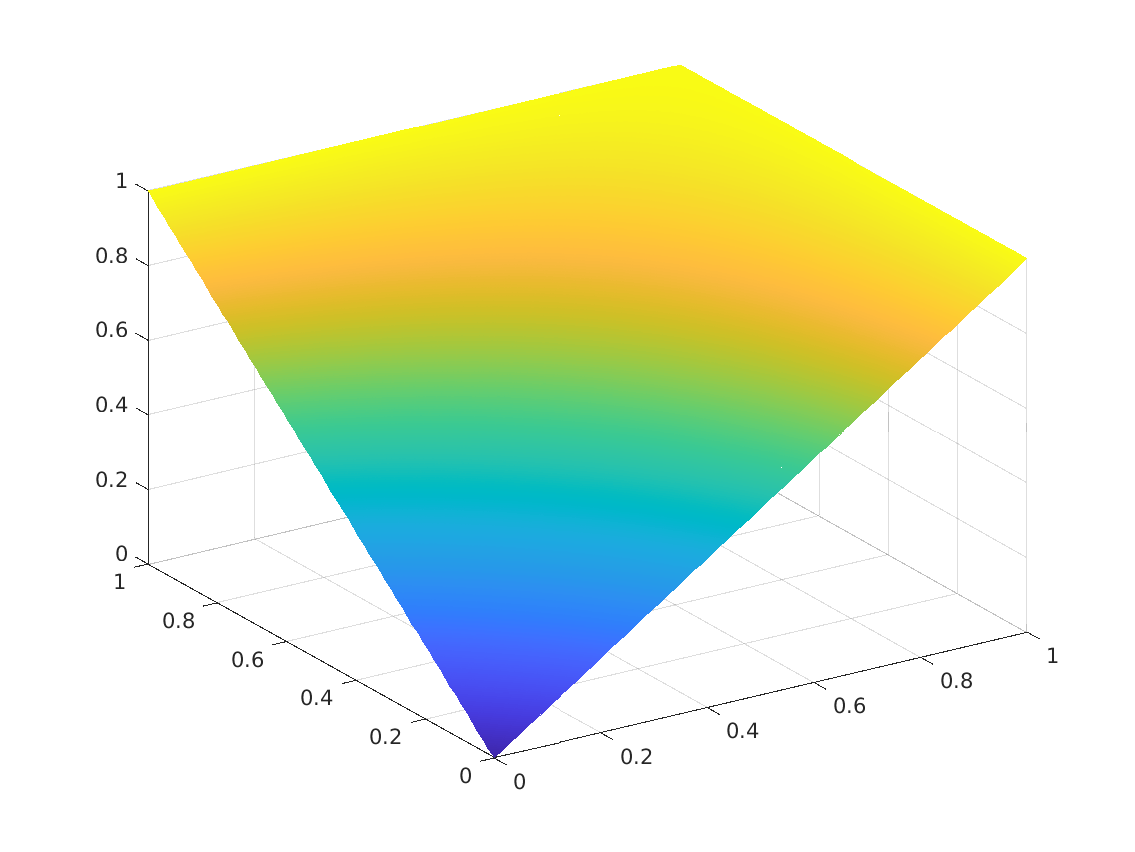}
    \end{minipage}%
    \begin{minipage}{.365\textwidth}
    \centering
      \includegraphics[width=1.30\linewidth]{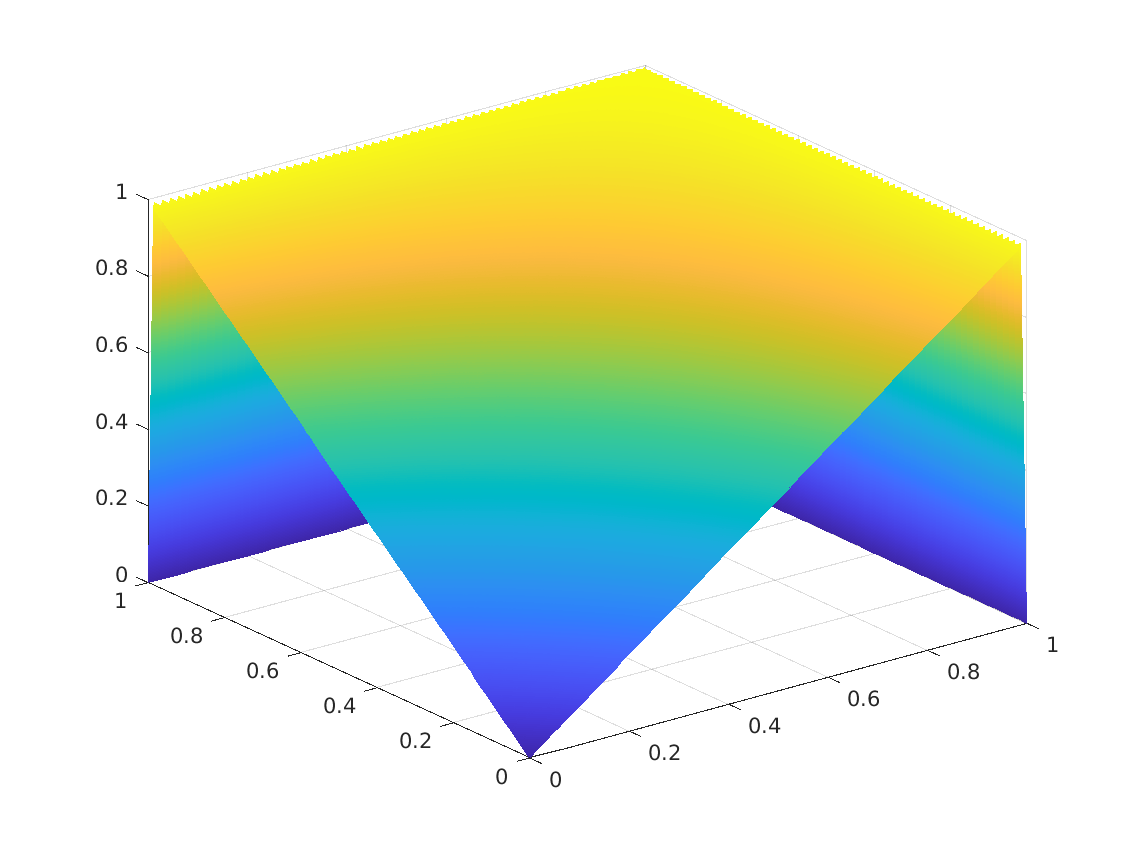}
    \end{minipage}
\end{figure}
\par As observed in Figure \ref{fig:ex1}, the numerical solution $u_h$ has no spurious oscillations in the convection-dominated regime. It is also obvious that the numerical solution $u_h$ ignores the boundary layers since we impose boundary conditions weakly.

Table \ref{Table1:ex1} shows the local convergence results on the subdomain $[0,0.875]^2$ in the $L_2$ norm, the $\|\cdot\|_h$ norm and the $\|\cdot\|_{h\sharp}$ norm. We observe $O(h^\frac32)$ convergence in the $\|\cdot\|_h$ and $\|\cdot\|_{h\sharp}$ norms as well as $O(h^2)$ convergence in the $L_2$ norm. Note that the local convergence behavior in the $L_2$ norm is optimal which indicates the boundary layer does not pollute the solution in the interior. 
\begin{table}[H]
    \centering
    \begin{tabular}{|c|c|cc|cc|cc|}
         \hline
        \multicolumn{1}{|c|}{}  
        & \multicolumn{1}{c|}{}
        & \multicolumn{2}{c|}{$L_2$}
        & \multicolumn{2}{c|}{$||\cdot||_h$}
        & \multicolumn{2}{c|}{$||\cdot||_{h\sharp}$}\\
        & $h$ & Error & Rate & Error & Rate & Error & Rate \\ 
        \hline
        \multirow{5}{*}{$\sigma_e=0$} 
        & 1/8  & 2.32e-04 & -- & 2.29e-03 & -- &  2.10e-02 & --      \\
        & 1/16 & 5.81e-05 & 2.00 & 8.08e-04 & 1.50 &  7.43e-03 & 1.50      \\
        & 1/32 & 1.45e-05 & 2.00 & 2.85e-04 & 1.50 &  2.63e-03 & 1.50      \\
        & 1/64 & 3.63e-06 & 2.00 & 1.00e-04 & 1.50 &  9.28e-04  & 1.50      \\
        \hline
        \hline
            \multirow{5}{*}{$\sigma_e=5$} 
        & 1/8  & 2.32e-04 & -- & 2.29e-03 & -- &  2.10e-02 & --      \\
        & 1/16 & 5.81e-05 & 2.00 & 8.08e-04 & 1.50 &  7.43e-03 & 1.50      \\
        & 1/32 & 1.45e-05 & 2.00 & 2.85e-04 & 1.50 &  2.63e-03 & 1.50      \\
        & 1/64 & 3.63e-06 & 1.99 & 1.00e-04 & 1.50 &  9.28e-04  & 1.50      \\
        \hline
    \end{tabular}
    \caption{Errors and rates of convergence for Example \ref{ex:bl1} on the subdomain $[0,0.875]^2$ (away from the boundary layer) when $\eps = 10^{-9}$.}
    \label{Table1:ex1}
\end{table}


In Table \ref{Table1:ex1global}, we show the global errors in the $L_2$ norm and the $\|\cdot\|_h$ norm on $\O = [0,1]^2$. We observe again the optimal convergence in the $L_2$ norm.  Notice that the global $\|\cdot\|_h$ errors do not converge at all due to the sharp boundary layer near the outflow boundary. Similar convergence behaviors were also observed in \cite{ayuso2009discontinuous}.

\begin{example}[Interior Layer \cite{ayuso2009discontinuous}]\label{ex:inl} 
    In this example, we take $\O = [0,1]^2$, $\gamma = 0$, $\bm{\zeta} = [\frac12, \frac{\sqrt{3}}{2}]^t$, $f = 0$, and the Dirichlet boundary conditions as:
\begin{align*} 
    u(x_1, x_2) =
        \begin{cases}
            1, & \text{on }\{x_2 = 0, 0 \leq x_1 \leq 1\},\\
            1, & \text{on }\{x_1 = 0, x_2 \leq \frac15\},\\
            0, & \text{elsewhere}.\\
        \end{cases}
\end{align*}
\end{example}

\begin{table}[H]
    \centering
    \begin{tabular}{|c|c|cc|cc|}
         \hline
        \multicolumn{1}{|c|}{}  
        & \multicolumn{1}{c|}{}
        & \multicolumn{2}{c|}{$L_2$}
        & \multicolumn{2}{c|}{$||\cdot||_h$} \\
        & $h$ & Error & Rate & Error & Rate  \\ 
        \hline
        \multirow{7}{*}{$\sigma_e=0$} 
       & 1/4 & 1.06e-03 & - & 1.00e+00 & -      \\
        & 1/8 & 2.66e-04 & 2.00 & 1.00e+00 & 0.00       \\
        & 1/16 & 6.64e-05 & 2.00 & 1.00e+00 & 0.00  \\
        & 1/32 & 1.66e-05 & 2.00 & 1.00e+00 & 0.00      \\
        & 1/64 & 4.15e-06 & 2.00 & 1.00e+00 & 0.00   \\
        \hline
        \hline
        \multicolumn{1}{|c|}{}  
        & \multicolumn{1}{c|}{}
        & \multicolumn{2}{c|}{$L_2$}
        & \multicolumn{2}{c|}{$||\cdot||_h$} \\
        & $h$ & Error & Rate & Error & Rate  \\ 
        \hline
        \multirow{7}{*}{$\sigma_e=5$} 
       & 1/4 & 1.06e-03 & - & 1.00e+00 & -      \\
        & 1/8 & 2.66e-04 & 1.99 & 1.00e+00 & 0.00       \\
        & 1/16 & 6.64e-05 & 2.00 & 1.00e+00 & 0.00  \\
        & 1/32 & 1.66e-05 & 2.00 & 1.00e+00 & 0.00      \\
        & 1/64 & 4.15e-06 & 2.00 & 1.00e+00 & 0.00   \\
        \hline
    \end{tabular}

    \caption{Errors and rates of convergence for $u_h$ on $\O = [0,1]^2$ for Example \ref{ex:bl1} when $\eps = 10^{-9}$ }
    \label{Table1:ex1global}
\end{table}

\begin{figure}[H]
\centering
\caption{Results of Example \ref{ex:inl}: $u_h$ (left) and the profile of $u_h$ at $x_1=0$ (right) for $\eps=10^{-9}$ and $h = \sfrac{1}{128}$.}
\label{fig:inl}
\begin{minipage}{.365\textwidth}
\hspace{-0.7in}
  \includegraphics[width=1.30\linewidth]{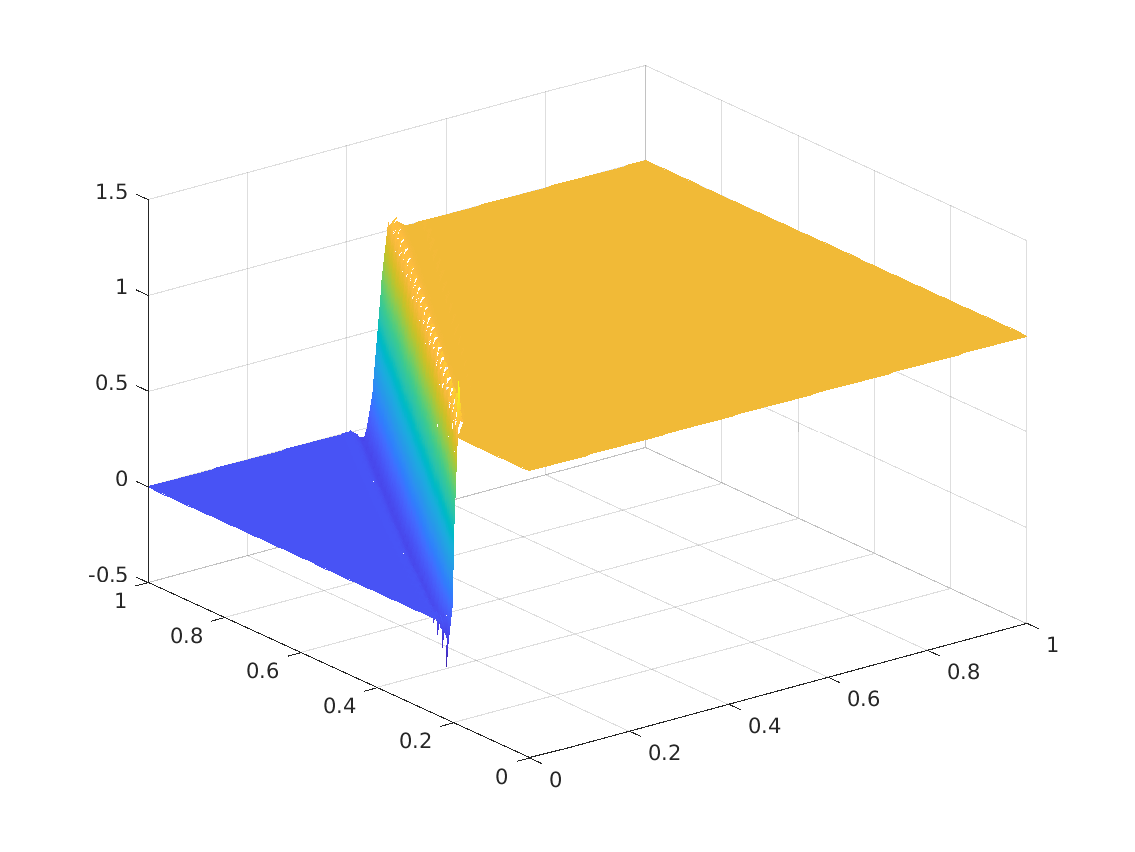}
\end{minipage}%
\begin{minipage}{.365\textwidth}
\centering
  \includegraphics[width=1.30\linewidth]{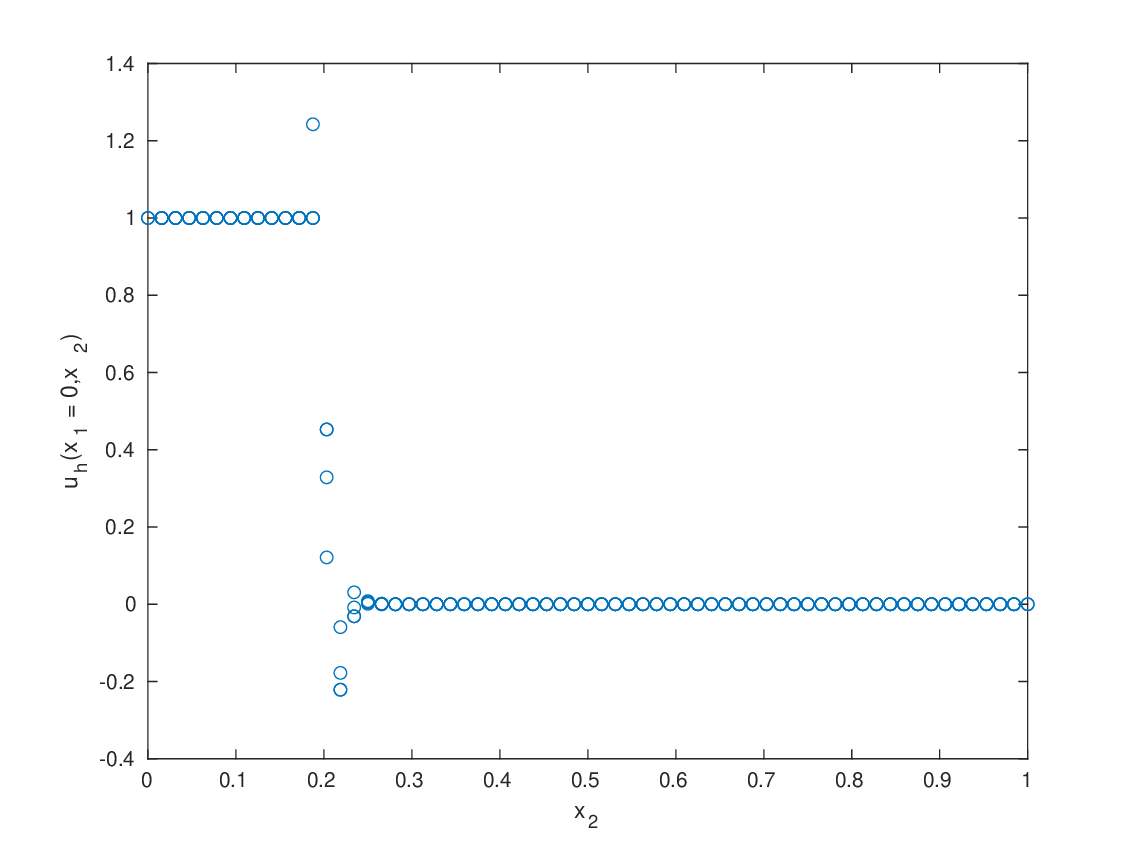}
\end{minipage}
\end{figure}

Figure \ref{fig:inl} shows that the presence of an internal layer in the approximate solution. Although the internal layer is captured by the approximate solution, there is small overshooting/undershooting along the internal layer. This is emphasized in the picture on the right in Figure \ref{ex:inl} where the profile of $u_h$ is plotted on $\{x_1:x_1=0\} \times \{x_2: x_2 \in [0,1]\}$. For comparison, we show the numerical solution for $\eps=10^{-3}$ in Figure \ref{fig:inle}. The behavior of our numerical methods is similar to standard DG methods (cf. \cite{ayuso2009discontinuous}). For example, there are wiggles near the outflow boundary in the intermediate regime and the boundary layer is ignored on the outflow boundary.

\begin{figure}[H]
    \centering
    \caption{Results of Example \ref{ex:inl}: $u_h$ (left) and the profile of $u_h$ at $x_1=0$ (right) for $\eps = 10^{-3}$ and $h = \sfrac{1}{128}$.}
    \label{fig:inle}
    \begin{minipage}{.365\textwidth}
    \hspace{-0.7in}
    \includegraphics[width=1.30\linewidth]{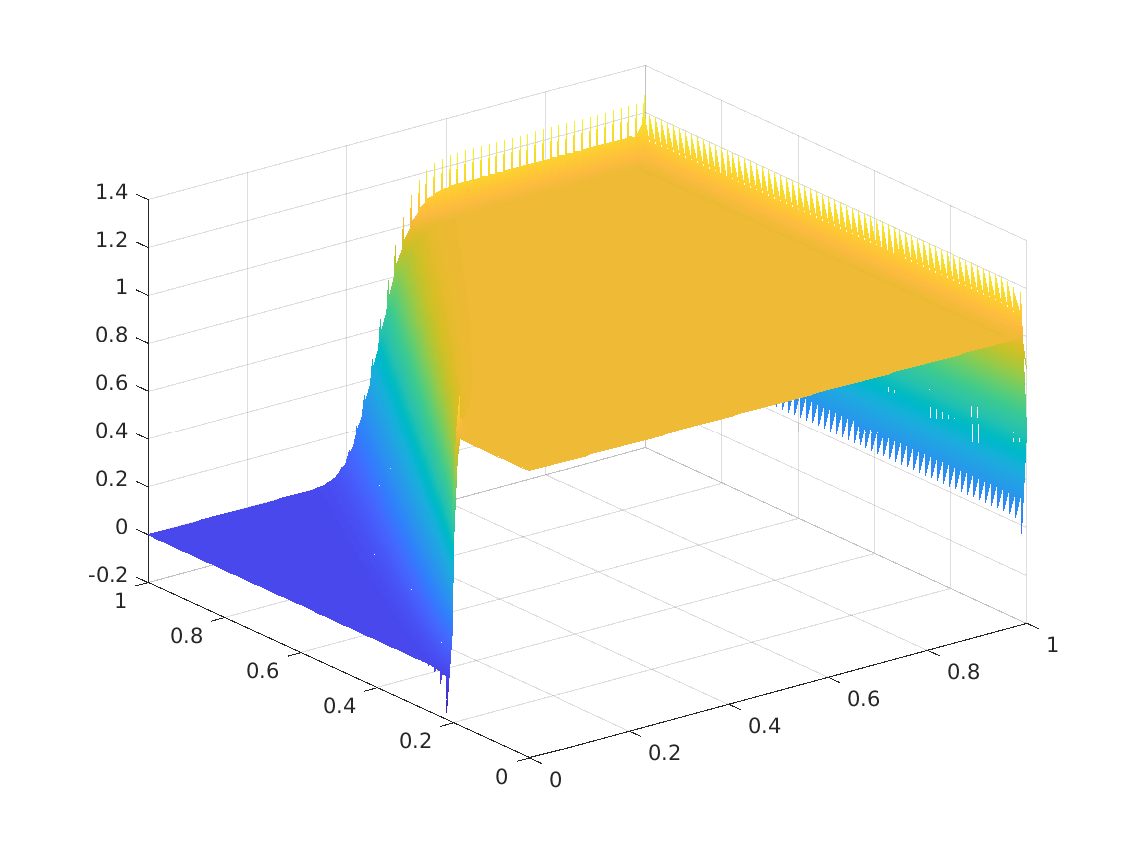}
    \end{minipage}%
    \begin{minipage}{.365\textwidth}
    \centering
    \includegraphics[width=1.30\linewidth]{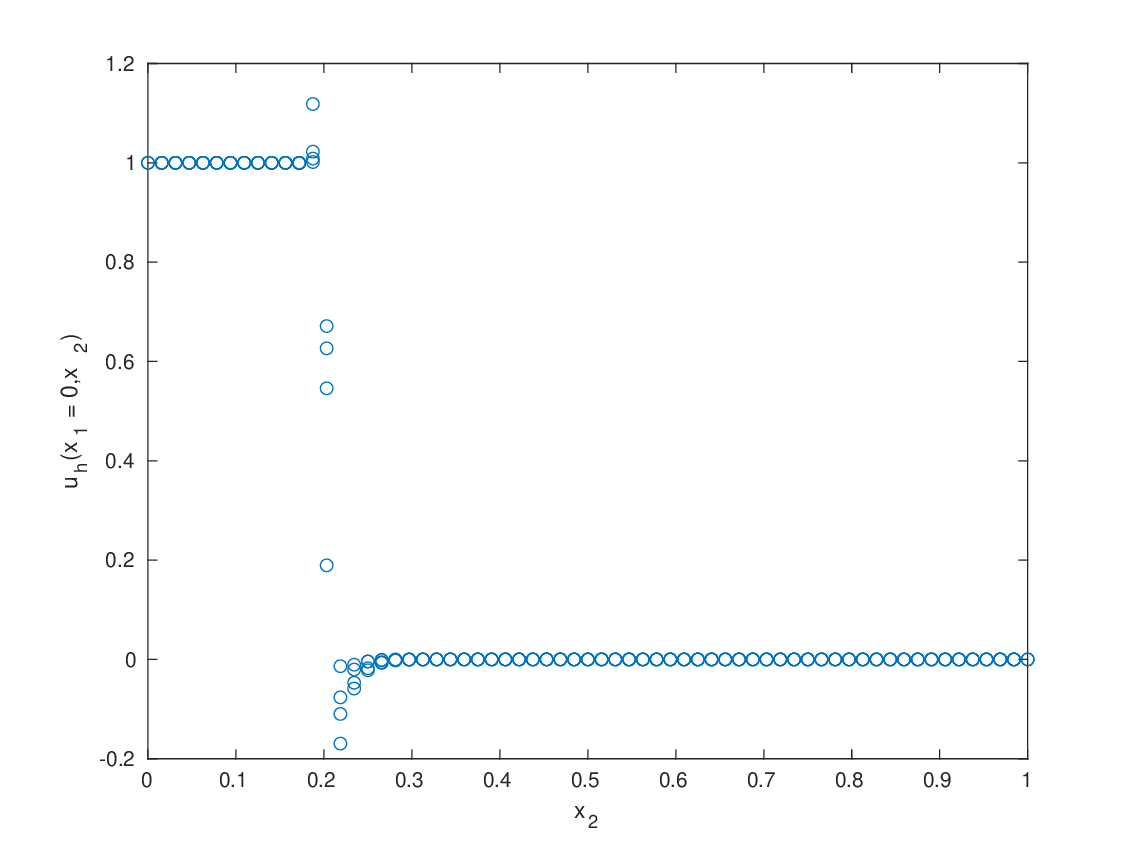}
    \end{minipage}
\end{figure}

\begin{example}[Interior Layer \cite{leykekhman2012local}]\label{ex:inl1}
    In this example, we take $\O = [0,1]^2$, $\gamma = 0$, $\bm{\zeta} = [1, 0]^t$, and $\eps = 10^{-9}.$ The exact solution is
\begin{align}\label{eq:inlexact}
    u(x_1,x_2) = (1-x_1)^3 \arctan\left(\dfrac{x_2 - 0.5}{\eps}\right).
\end{align}
\end{example}

\begin{figure}[H]
    \centering
    \caption{Results of Example \ref{ex:inl1}: $u_h$ (left) and $u$ (right) for $\eps = 10^{-9}$ and $h = \sfrac{1}{128}$.}
    \label{fig:inl1}
    \begin{minipage}{.365\textwidth}
    \hspace{-0.7in}
      \includegraphics[width=1.30\linewidth]{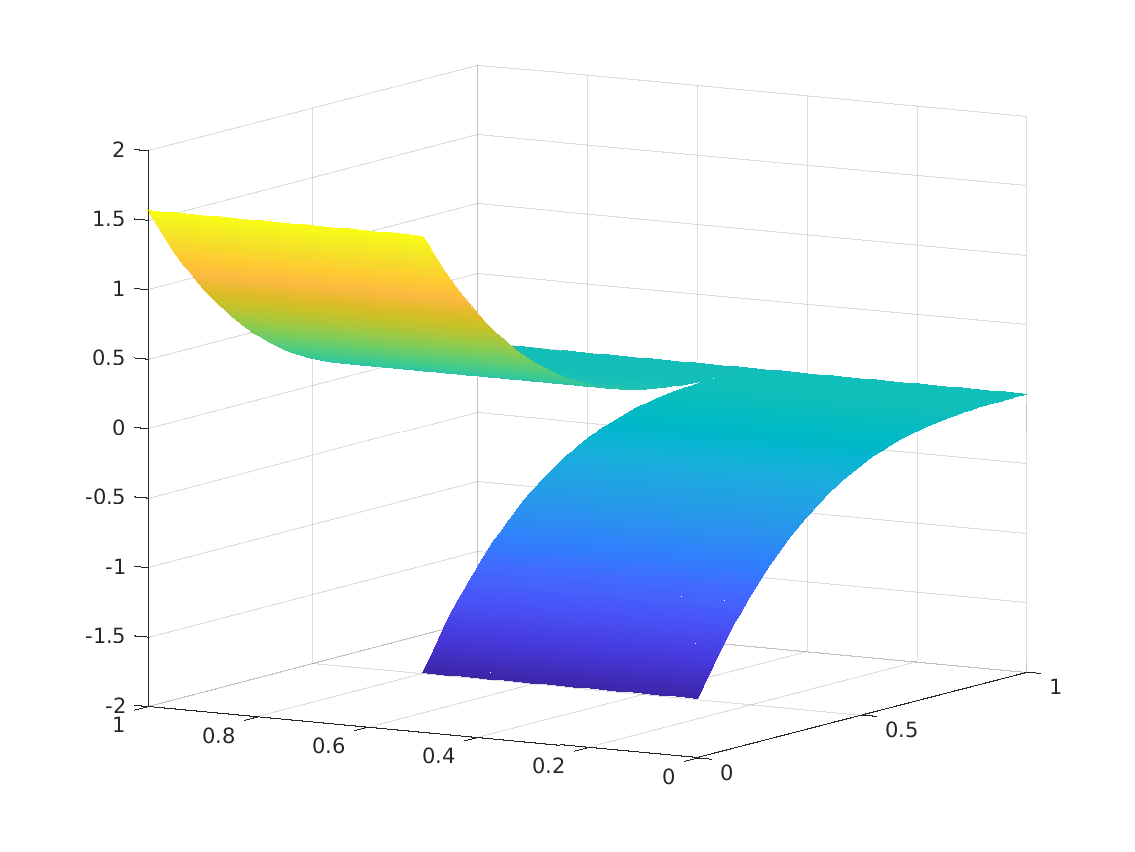}
    \end{minipage}%
    \begin{minipage}{.365\textwidth}
    \centering
      \includegraphics[width=1.30\linewidth]{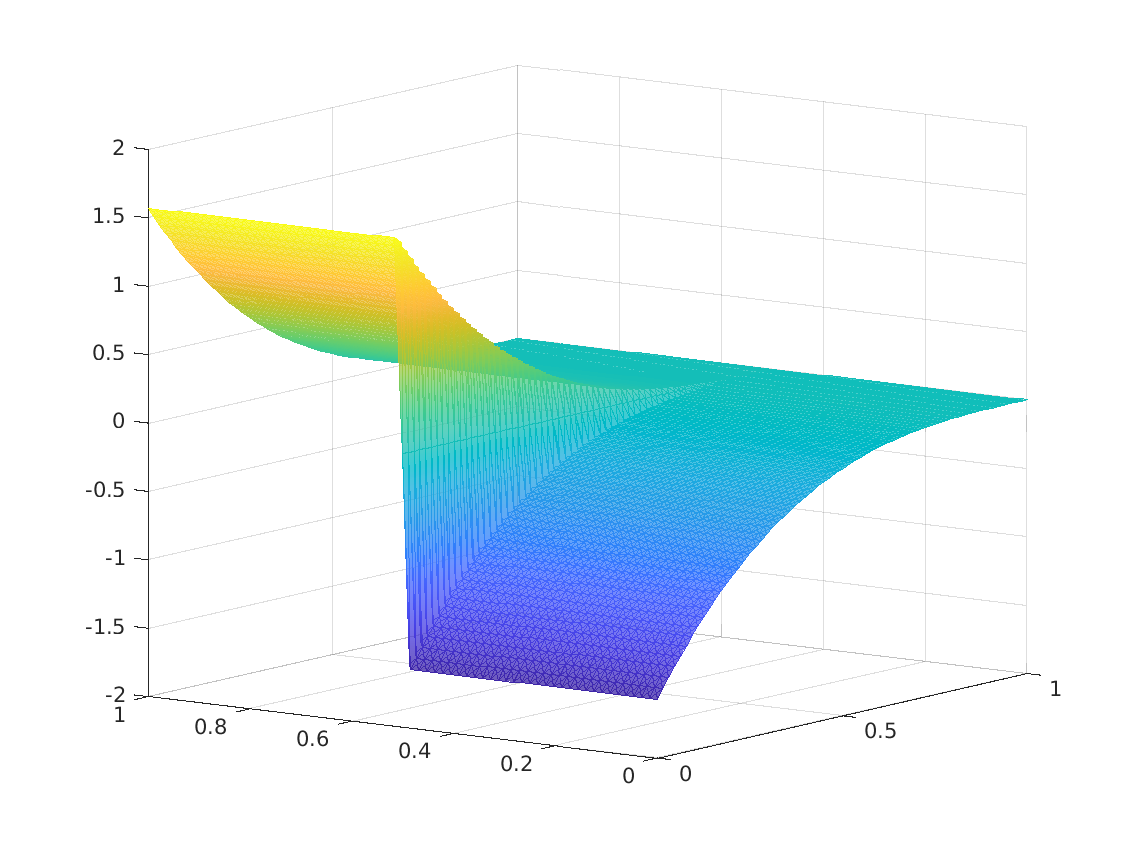}
    \end{minipage}
\end{figure}

As one can see from Figure \ref{fig:inl1}, the exact solution $u$ has an internal layer along $x_2 = 0.5.$ It is also clear that the numerical solution does not resolve the interior layer (cf. \cite{leykekhman2012local}).

In Table \ref{Table:inl1local}, we compute the local convergence in the $L_2$ norm and the $\|\cdot\|_h$ norm. We again observe $O(h^2)$ convergence  in the $L_2$ norm and $O(h^\frac32)$ convergence in the $\|\cdot\|_h$ and $\|\cdot\|_h\sharp$ norms. One can see that the convergence rates are optimal in the region where the solution is smooth. This indicates that the interior layer does not pollute the solution into the region that stays away from the interior layer.

\begin{table}[H]
    \centering
    \begin{tabular}{|c|c|cc|cc|cc|}
         \hline
        \multicolumn{1}{|c|}{}  
        & \multicolumn{1}{c|}{}
        & \multicolumn{2}{c|}{$L_2$}
        & \multicolumn{2}{c|}{$||\cdot||_h$}
        & \multicolumn{2}{c|}{$||\cdot||_{h\sharp}$}\\
        & $h$ & Error & Rate & Error & Rate & Error & Rate \\ 
        \hline
        \multirow{5}{*}{$\sigma_e=0$} 
        & 1/8  & 9.57e-04 & -- & 8.16e-03 & -- &  5.78e-02 & --      \\
        & 1/16 & 2.42e-04 & 1.98 & 2.88e-03 & 1.50 &  2.04e-02 & 1.49      \\
        & 1/32 & 6.10e-05 & 1.99 & 1.02e-03 & 1.50 &  7.23e-03 & 1.50      \\
        & 1/64 & 1.53e-05 & 2.00 & 3.59e-04 & 1.50 &  2.55e-03  & 1.50      \\
        \hline
        \hline
            \multirow{5}{*}{$\sigma_e=5$} 
        & 1/8  & 9.57e-04 & -- & 8.16e-03 & -- &  5.78e-02 & --      \\
        & 1/16 & 2.42e-04 & 1.98 & 2.88e-03 & 1.50 &  2.04e-02 & 1.49      \\
        & 1/32 & 6.10e-05 & 1.99 & 1.02e-03 & 1.50 &  7.23e-03 & 1.50      \\
        & 1/64 & 1.53e-05 & 2.00 & 3.59e-04 & 1.50 &  2.55e-03  & 1.50      \\
        \hline
    \end{tabular}
    \caption{Errors and rates of convergence for $u_h$ for Example \ref{ex:inl1} in the subdomain $[0,1] \times [0.625,1]$ (away from the interior layer) when $\eps = 10^{-9}$.}
    \label{Table:inl1local}
\end{table}

For comparison, we show the global convergence rates in Table \ref{Table:inl1global}. We see that the convergence rates deteriorate when $h$ is small due to the interior layer. We also illustrate the behavior of our numerical methods in Figures \ref{fig:inl11} and \ref{fig:inl12} when $\eps=10^{-3}$ and $\eps=1$. We can clearly see our methods capture the interior layer when $\eps$ increases. 

\section{Concluding Remarks}

In this paper we developed and analyzed numerical approximations based on the DG finite element differential calculus framework for a convection-diffusion-reaction equation. We proved that the proposed methods have optimal convergence behaviors in the convection-dominated regime. As a byproduct, we also showed that the method for the reduced convection-reaction problem is equivalent to a centered fluxes DG method. Numerically, we also observed that our methods have optimal convergence rates in the interior of the domain which are away from the boundary layers and interior layers. An interesting problem is to extend our methods to an optimal control problem that is constrained by a convection-dominated equation (cf. \cite{liu2024discontinuous,liu2024robust}). This is being investigated in an ongoing project.

\begin{table}[H]
    \centering
    \begin{tabular}{|c|c|cc|cc|}
         \hline
        \multicolumn{1}{|c|}{}  
        & \multicolumn{1}{c|}{}
        & \multicolumn{2}{c|}{$L_2$}
        & \multicolumn{2}{c|}{$||\cdot||_h$} \\
        & $h$ & Error & Rate & Error & Rate  \\ 
        \hline
        \multirow{7}{*}{$\sigma_e=0$} 
        & 1/4 & 6.06e-03 & - & 3.77e-02 & -      \\
        & 1/8 & 1.56e-04 & 1.96 & 1.33e-02 & 1.50       \\
        & 1/16 & 3.96e-04 & 1.98 & 4.73e-03 & 1.49  \\
        & 1/32 & 9.97e-05 & 1.99 & 1.82e-03 & 1.38      \\
        & 1/64 & 2.72e-05 & 1.87 & 1.19e-03 & 0.60   \\
        \hline
        \hline
        \multicolumn{1}{|c|}{}  
        & \multicolumn{1}{c|}{}
        & \multicolumn{2}{c|}{$L_2$}
        & \multicolumn{2}{c|}{$||\cdot||_h$} \\
        & $h$ & Error & Rate & Error & Rate  \\ 
        \hline
        \multirow{7}{*}{$\sigma_e=5$} 
       & 1/4 & 6.06e-03 & - & 3.77e-02 & -     \\
        & 1/8 & 1.56e-04 & 1.96 & 1.33e-02 & 1.50       \\
        & 1/16 & 3.96e-04 & 1.98 & 4.74e-03 & 1.49  \\
        & 1/32 & 9.97e-05 & 1.99 & 1.88e-03 & 1.33      \\
        & 1/64 & 2.95e-05 & 1.76 & 1.37e-03 & 0.45   \\
        \hline
    \end{tabular}
    \caption{Errors and rates of convergence for Example \ref{ex:inl1} on $\O = [0,1]^2$ for $\eps = 10^{-9}$.}
    \label{Table:inl1global}
\end{table}

\begin{figure}[H]
    \centering
    \caption{Results of Example \ref{ex:inl1}: $u_h$ (left) and $u$ (right) for $\eps = 10^{-3}$ and $h = \sfrac{1}{128}$.}
    \label{fig:inl11}
    \begin{minipage}{.365\textwidth}
    \hspace{-0.7in}
    \includegraphics[width=1.30\linewidth]{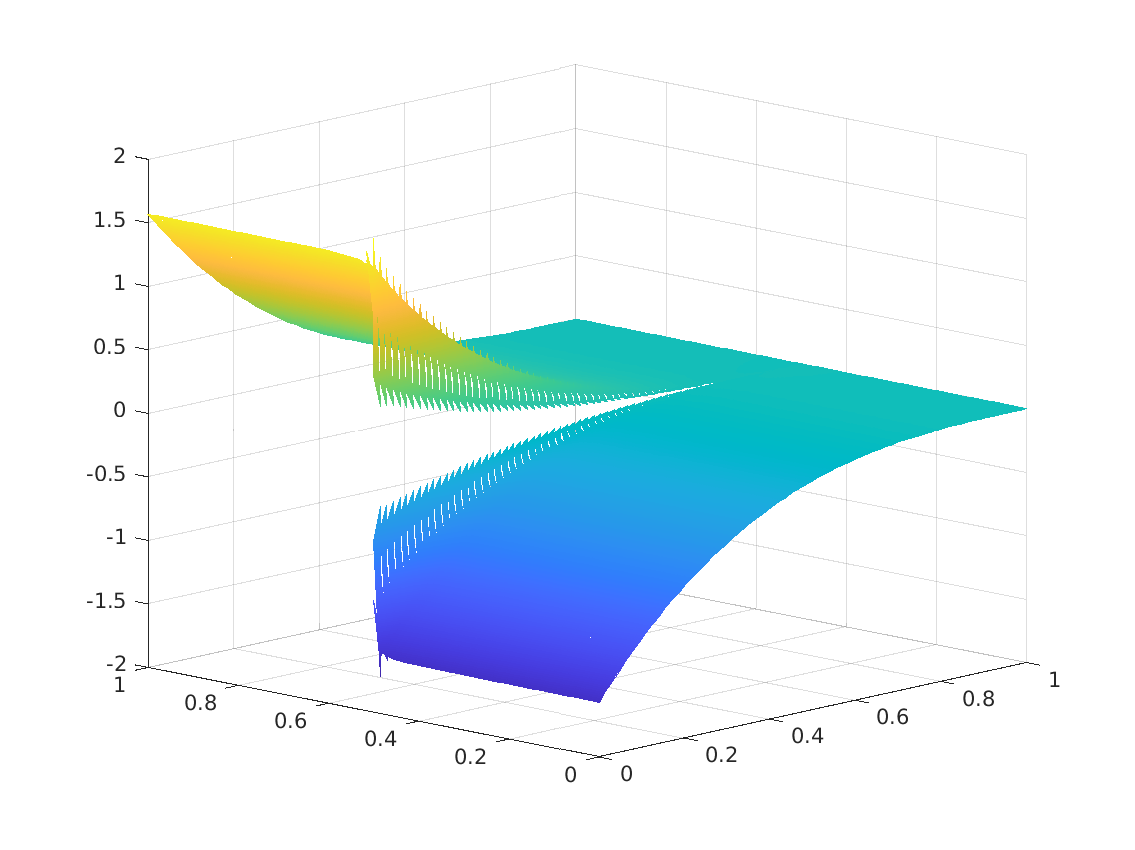}
    \end{minipage}%
    \begin{minipage}{.365\textwidth}
    \centering
    \includegraphics[width=1.30\linewidth]{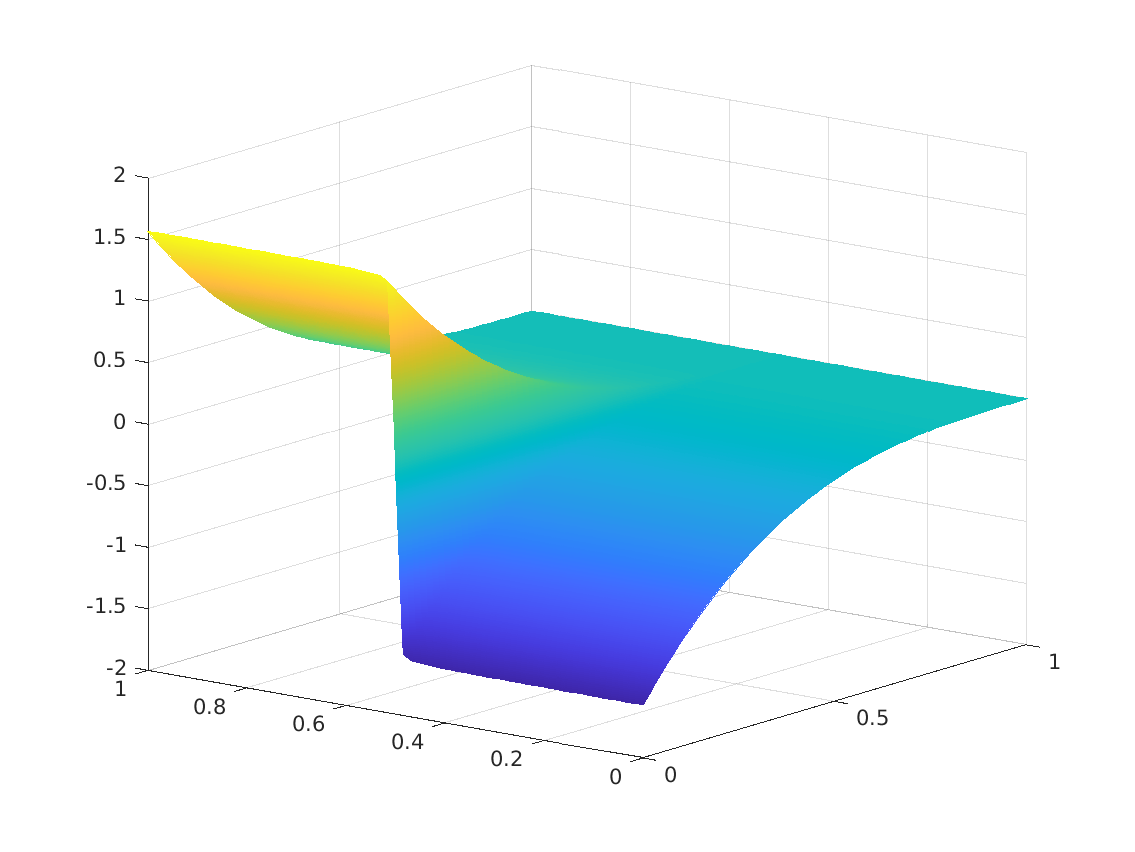}
    \end{minipage}
\end{figure}

\section*{Acknowledgement}
This material is based upon work supported by the National Science Foundation under Grants No. DMS-2111059, DMS-2111004, and DMS-1929284 while the third author was in residence at the Institute for Computational and Experimental Research in Mathematics in Providence, RI, during the "Numerical PDEs: Analysis, Algorithms, and Data Challenges" program.

\begin{figure}[H]
    \centering
    \caption{Results of Example \ref{ex:inl1}: $u_h$ (left) and $u$ (right) for $\eps = 1$ and $h = \sfrac{1}{128}$.}
    \label{fig:inl12}
    \begin{minipage}{.365\textwidth}
    \hspace{-0.7in}
    \includegraphics[width=1.30\linewidth]{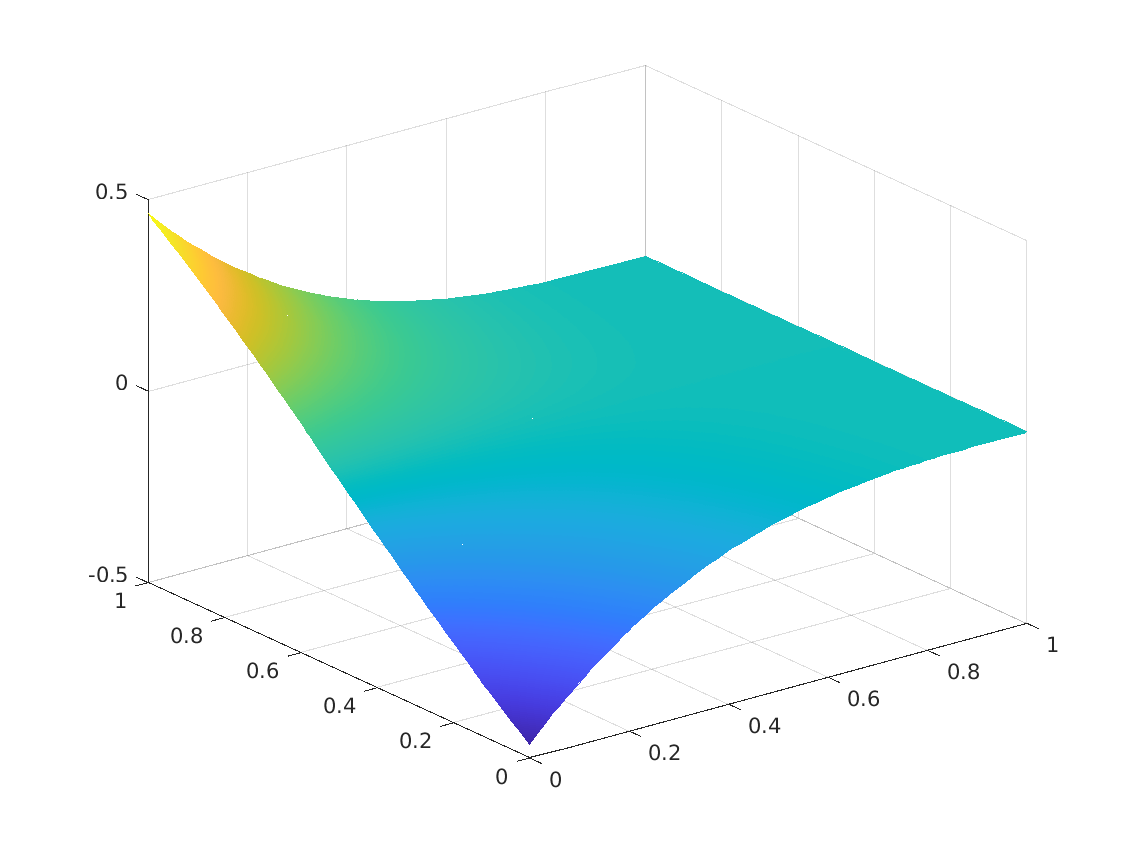}
    \end{minipage}%
    \begin{minipage}{.365\textwidth}
    \centering
    \includegraphics[width=1.30\linewidth]{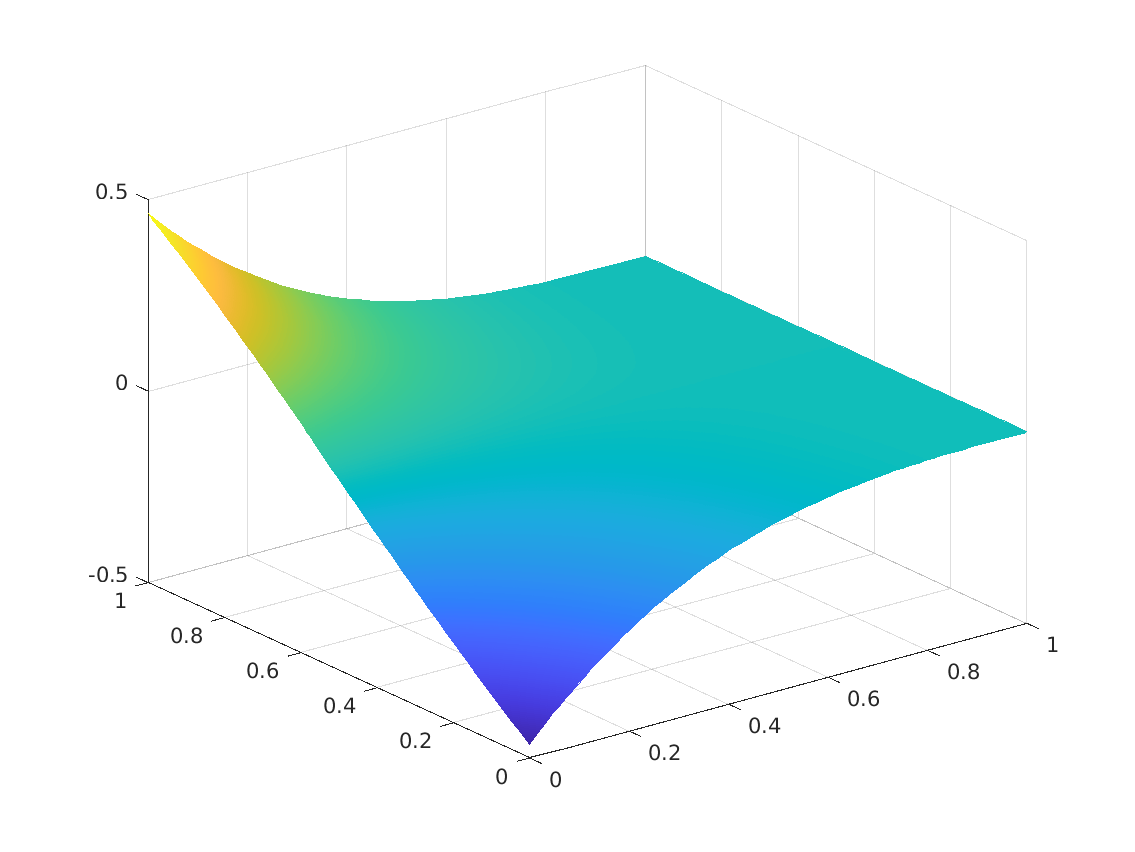}
    \end{minipage}
\end{figure}

\appendix

\section{Proof of Lemma \ref{lem:infsupconvah}}\label{apx:infsup}

 \begin{proof}[Proof of Lemma \ref{lem:infsupconvah}]
    We follow the approaches in \cite{di2011mathematical,ern2006discontinuous}. Let $\mathcal{S}=\sup_{w_h\in V_h\setminus\{0\}}\frac{a_h(v_h,w_h)}{\|w_h\|_{h\sharp}}$. Given any $v_h\in V_h$, we construct a particular $w_h\in V_h\setminus\{0\}$ such that, for all $T\in\cT_h$, $w_h|_T=h_T\l\bm{\zeta}\r_T\cdot\nabla v_h$, where $\l\bm{\zeta}\r_T$ denotes the mean value of $\bm{\zeta}$ over $T$. We first notice that, by \eqref{eq:ahcoer},
    \begin{equation}\label{eq:ahcoercon}
        C\|v_h\|_h^2\le a_h(v_h,v_h)=\frac{a_h(v_h,v_h)}{\|v_h\|_{h\sharp}}\|v_h\|_{h\sharp}\le \mathcal{S}\|v_h\|_{h\sharp}.
    \end{equation}
    We claim that
    \begin{equation}\label{eq:ahinfsupex1}
        \sum_{T\in\cT_h}h_T\|\bm{\zeta}\cdot\nabla v_h\|^2_{L_2(T)}\lesssim \mathcal{S}\|v_h\|_{h\sharp}+\|v_h\|_h\|v_h\|_{h\sharp}+\|v_h\|_h^2.
    \end{equation}
    Combining \eqref{eq:ahinfsupex1} and \eqref{eq:ahcoercon} and using \eqref{eq:ahcoercon} again, we have
    \begin{equation}
        C\|v_h\|^2_{h\sharp}\lesssim \mathcal{S}\|v_h\|_{h\sharp}+\|v_h\|_h\|v_h\|_{h\sharp}.
    \end{equation}
    Upon using Young's inequality and iterating the inequality \eqref{eq:ahcoercon} once again, we have
    \begin{equation}
        C\|v_h\|^2_{h\sharp}\lesssim \mathcal{S}\|v_h\|_{h\sharp}
    \end{equation}
    which leads to \eqref{eq:disinfsuph}. The rest of the proof is devoted to \eqref{eq:ahinfsupex1}. We first prove the estimate
    \begin{equation}\label{eq:whvh}
        \|w_h\|_{h\sharp}\lesssim\|v_h\|_{h\sharp}.
    \end{equation}
    Indeed, it follows from \eqref{eq:upwnormstrongah} that
    \begin{equation}
    \begin{aligned}
        \|w_h\|^2_{h\sharp}&=\eps\|w_h\|_d^2+\|w_h\|_{upw}^2+\sum_{T\in\cT_h}h_T\|\bm{\zeta}\cdot\nabla w_h\|^2_{L_2(T)}.
    \end{aligned}
    \end{equation}
    A standard inverse inequality implies
    \begin{equation}\label{eq:infsupineqinv}
        \sum_{T\in\cT_h}h_T\|\bm{\zeta}\cdot\nabla w_h\|^2_{L_2(T)}\lesssim \sum_{T\in\cT_h}h_T^{-1}\|w_h\|^2_{L_2(T)}
    \end{equation}
    and, together with a trace inequality 
    \begin{equation}\label{eq:infsupineq}
    \begin{aligned}
        \|w_h\|^2_{upw}&=\|w_h\|^2_{\LT}+\int_{\partial\O} \frac12|\bm{\zeta}\cdot\bn|w_h^2\ \!ds+\sum_{e\in\cE_h^I}\int_e \frac12|\bm{\zeta}\cdot\bn|[w_h]^2\ \!ds\\
        &\lesssim \|v_h\|_\LT^2+\sum_{T\in\cT_h}h_T^{-1}\|w_h\|^2_{L_2(T)}.
    \end{aligned}
    \end{equation}
    We also have $\|w_h\|_d^2\lesssim \|v_h\|_d^2$. In fact, we have, if $\sigma_{min}:=\min_{e\in \cE_h}\sigma_e>0$,
    \begin{equation}\label{eq:penaltynormeq}
    \begin{aligned}
        \sum_{e\in\cE_h}\frac{\sigma_e}{h_e}\|[w_h]\|_{L_2(e)}^2\le C\sum_{T\in\cT_h}\|\nabla v_h\|_{L_2(T)}^2\le C(1+\frac{1}{\sigma_{min}})\|v_h\|_d^2,
    \end{aligned}
    \end{equation} 
    where we use a standard trace inequality and \cite[Lemma 4.1]{lewis2020convergence}.
    It also follows from \cite[Lemma 4.1]{lewis2020convergence} and a trace inequality that, for $\sigma_e\ge0$,  
    \begin{equation}\label{eq:avehnormeq}
    \begin{aligned}
        &\frac12(\|\nabla^+_{h,0}w_h\|_{\LT}^2+\|\nabla^-_{h,0}w_h\|_{\LT}^2)\\
        \lesssim& \sum_{T\in\cT_h}\|\nabla w_h\|^2_{L_2(T)}+\sum_{e\in\cE_h}\frac{1}{h_e}\|[w_h]\|_{L_2(e)}^2\\
        \lesssim&\|v_h\|_d^2.
    \end{aligned}
    \end{equation}
    The estimates \eqref{eq:penaltynormeq} and \eqref{eq:avehnormeq} then imply $\|w_h\|_d^2\lesssim \|v_h\|_d^2$.

    At last, it is known that \cite{di2011mathematical}
    \begin{equation}\label{eq:infsupineqs}
        \sum_{T\in\cT_h}h_T^{-1}\|w_h\|^2_{L_2(T)}\lesssim \|v_h\|^2_{h\sharp}.
    \end{equation}
    The estimate \eqref{eq:whvh} is immediate upon combining \eqref{eq:infsupineqinv}-\eqref{eq:infsupineqs}.

    It follows from \eqref{eq:dwdgdfbilinear}, \eqref{eq:cf}, and \eqref{eq:dwdgstable} that
    \begin{equation}\label{eq:infsuphtnorm}
    \begin{aligned}
        \sum_{T\in\cT_h}h_T\|\bm{\zeta}\cdot\nabla v_h\|^2_{L_2(T)}&=a_h(v_h,w_h)-\eps a_h^{d}(v_h,w_h)\\
        &\quad+(\bm{\zeta}\cdot\nabla v_h,h_T(\bm{\zeta}-\l\bm{\zeta}\r_T)\cdot\nabla v_h)_{\cT_h}\\
        &\quad+\l \bm{\zeta}\cdot\bn[v_h],\{w_h\}\r_{\cE^I_h}-(\gamma v_h,w_h)_\LT\\
        &\quad-\int_{\partial\O^-}|\bm{\zeta}\cdot\bn| v_h\ \!w_h\ \! ds-\l \frac12|\bm{\zeta}\cdot\bn|[v_h],[w_h]\r_{\cE^I_h}\\
        &=T_1+T_2\cdots+T_7.
    \end{aligned}
    \end{equation}
    For the first two terms, we have, by \eqref{eq:whvh} and \eqref{eq:dwdgbounded},
    \begin{equation}\label{eq:T1}
        |T_1|=|\frac{a_h(v_h,w_h)}{\|w_h\|_{h\sharp}}\|w_h\|_{h\sharp}|\le \mathcal{S}\|w_h\|_{h\sharp}\lesssim \mathcal{S}\|v_h\|_{h\sharp},
    \end{equation}
    \begin{equation}\label{eq:T2}   
        |T_2|\lesssim \eps^\frac12\|v_h\|_d\eps^\frac12\|w_h\|_d\lesssim \|v_h\|_h\|v_h\|_{h\sharp}.
    \end{equation}
    It follows from Cauchy-Schwarz inequality and \eqref{eq:whvh} that
    \begin{equation}\label{eq:T567}
        |T_5|+|T_6|+|T_7|\lesssim\|v_h\|_{h}\|w_h\|_{h\sharp}\lesssim\|v_h\|_{h}\|v_h\|_{h\sharp}.
    \end{equation}
    To bound $T_4$, we have, by a standard trace inequality, \eqref{eq:infsupineqs}, and \eqref{eq:whvh},
    \begin{equation}\label{eq:T_4}
    \begin{aligned}
        &\l\bm{\zeta}\cdot\bn[v_h],\{w_h\}\r_{\cE_h^I}&\\
        &\le\left(\sum_{e\in\cE_h^I}\int_e \frac12|\bm{\zeta}\cdot\bn|[v_h]^2\ \!ds\right)^\frac12\left(\sum_{e\in\cE_h^I}\int_e 2|\bm{\zeta}\cdot\bn|\{w_h\}^2\ \!ds\right)^\frac12\\
        &\lesssim\|v_h\|_{h}\left(\sum_{T\in\cT_h}h_T^{-1}\|w_h\|_{L_2(T)}^2\right)^\frac12\lesssim\|v_h\|_h\|v_h\|_{h\sharp}.
    \end{aligned}
    \end{equation}
    Finally, we bound $T_3$ as follows,
    \begin{equation}\label{eq:T3}
        \begin{aligned}
            (\bm{\zeta}\cdot\nabla &v_h,h_T(\bm{\zeta}-\l\bm{\zeta}\r_T)\cdot\nabla v_h)_{\cT_h}\\
            &\lesssim\left(\sum_{T\in\cT_h}h_T\|\bm{\zeta}\cdot\nabla v_h\|_{L_2(T)}^2\right)^\frac12\left(\sum_{T\in\cT_h}h_T\|v_h\|_{L_2(T)}^2\right)^\frac12\\
            &\lesssim\left(\sum_{T\in\cT_h}h_T\|\bm{\zeta}\cdot\nabla v_h\|_{L_2(T)}^2\right)^\frac12\|v_h\|_{h\sharp}\\
            &\lesssim\frac12\sum_{T\in\cT_h}h_T\|\bm{\zeta}\cdot\nabla v_h\|_{L_2(T)}^2+C\|v_h\|^2_{h\sharp},
        \end{aligned}
    \end{equation}
    where we use an inverse inequality and Young's inequality. We also use the fact $\bm{\zeta}\in [W^{1,\infty}(\Omega)]^2$, and, hence, $\|\bm{\zeta}-\l\bm{\zeta}\r_T\|_{L^\infty(T)}\lesssim h_T$.
    The claimed estimate \eqref{eq:ahinfsupex1} follows from \eqref{eq:infsuphtnorm}-\eqref{eq:T3}.
 \end{proof}

\bibliographystyle{plain}
\bibliography{references}
 
\end{document}